\documentclass{amsart}

\usepackage[utf8]{inputenc}
\usepackage[T1]{fontenc}
\usepackage{amssymb}
\usepackage[english]{babel}
\usepackage{amsmath}
\usepackage{amsthm}
\usepackage{epsfig}
\usepackage{mathrsfs}
\usepackage{xcolor}
\usepackage{graphicx}
\usepackage{esint}
\usepackage{enumerate}
\usepackage{cases}
\usepackage[toc,page]{appendix}
\usepackage{hyperref}
\usepackage{xcolor}
\usepackage{verbatim}

\usepackage{todonotes}\reversemarginpar

\newtheorem{theorem}{Theorem}[section]
\newtheorem{lemma}[theorem]{Lemma}

\newtheorem{proposition}[theorem]{Proposition}
\newtheorem{corollary}[theorem]{Corollary}

\newtheorem{assumption}{Assumption}

\def\eps{\varepsilon}

\def\Tr{\mathrm{Tr}}
\def\pa{\partial}
\def\na{\nabla}
\def\RR{\mathbb R}
\def\NN{\mathbb N}
\def\R{\mathbb R}
\def\O{\mathcal O}

\def\Graph{ \mathrm{Graph}}
\def\supp{\mathrm {spt}\,}

\title[Regularity of optimal transportation potentials with rough measures]{H\"older and Sobolev regularity of optimal transportation potentials with rough measures}
\author{Pierre-Emmanuel Jabin}
\thanks{pejabin@psu.edu, Pennsylvania State University, Department of Mathematics and Huck Institutes, State College, PA 16802. Partially supported by NSF DMS Grant 161453, 1908739, 2049020 and NSF Grant RNMS (Ki-Net) 1107444.}
\author{Antoine Mellet}
\thanks{mellet@umd.edu, University of Maryland, Department of Mathematics and CSCAMM, College Park, MD 20742  USA. Partially supported by NSF Grant NSF Grant  DMS-2009236 and DMS-2307342.}

\date{}

\begin{document}
\maketitle

\begin{abstract}
We consider a Kantorovich potential associated to an optimal transportation problem between measures that are not necessarily absolutely continuous with respect to the Lebesgue measure, but are comparable to the Lebesgue measure when restricted to balls with radius greater than some $\delta>0$. 
Our main results extend the classical regularity theory of optimal transportation to this framework. In particular, we establish both H\"older and Sobolev regularity results for Kantorovich potentials up to some critical length scale depending on $\delta$.
Our assumptions are very natural in the context of the numerical computation of optimal maps, which often involves approximating by sums of Dirac masses some measures that are  absolutely continuous with densities bounded away from zero and infinity on their supports.
\end{abstract}

\section{Introduction}
\subsection{Optimal transportation with general measure}
The theory of optimal transportation, first introduced  by Monge in the late 1800 (and later generalized by Kantorovich), has 
been extensively studied from a theoretical point of view in the past few decades.
It has also found numerous applications to problems both theoretical and applied. 
Following the work of Brenier \cite{Brenier87}, a major topic of mathematical research has been the study of the regularity 
of the optimal transportation map between absolutely continuous measures. 

More recently, partly motivated by numerical applications, there has been a lot of interest in optimal transportation problems between general measures. Indeed, numerical schemes often involve measures that are sums of Dirac masses. Solutions of optimal transportation problems between such singular measures cannot be expected to be smooth in general. However it is natural to expect these solutions to have some nice properties when the general measures, possibly discrete, are close (in some sense) to some nice measures (see Assumption \ref{ass:1}). 
For example, if some measures are obtained by discretizing some smooth densities at a scale $\delta$, they ``look'' smooth when seen only at scales larger than $\delta$. One can then ask whether the solution of the corresponding optimal transportation problem also looks smooth when seen at a correspondingly large enough scale (whose relation to $\delta$ needs to be identified).

This article introduces a general strategy to extend the classical  regularity theory for optimal transportation problem between continuous densities to such a discrete setting. The main take away of our analysis is that the solutions of this discrete optimal transportation problem have the usual Sololev regularity when mollified at a scale $\delta^{1/2}$
(see Theorems \ref {thm:2} and \ref{thm:sobolev}) - this scale is optimal in the sense that it is consistent with known convergence rates of numerical schemes (see Section \ref{sec:intro3}).

Let us now briefly recall the setting of optimal transportation with general measures.
Throughout the paper, we will focus  exclusively on  the classical quadratic cost (we refer for example to \cite{villani2008optimal} for a detailed presentation of the topic in a much more general framework):

Given two convex bounded open sets $\Omega$ and $\O$ in $\R^n$ ($n\geq 1$), we  consider two Borel probability measures $\mu$, $\nu$ such that
$ \supp \mu \subset \overline \Omega$ and $\supp \nu\subset\overline \O$.
In Kantorovich's setting, admissible transference plans are probability measures on $\R^n\times\R^n$ with marginals $\mu$ and $\nu$. We denote by $\Pi(\mu,\nu) $ the set of transference plans:
$$
\Pi(\mu,\nu): = \{ \pi \in \mathcal P(\R^n\times\R^n) \, ;\, \pi(A\times\R^n)= \mu(A), \; \pi(\R^n\times B) =\nu(B)\}
$$
(where $A$ and $B$ are any measurable sets in $\R^n$).
The classical optimal transportation problem (with quadratic cost), is concerned with the minimization problem
\begin{equation}\label{eq:OT}
C(\mu,\nu) = \inf_{\pi\in \Pi(\mu,\nu)} \int_{\R^n\times\R^n} |x-z|^2 d\pi(x,z).
\end{equation}
The existence of a minimizer for \eqref{eq:OT} is given by the following classical result (see for instance \cite[Theorem 2.12]{villani2003topics}):
\begin{theorem}\label{thm:OT1}
If $C(\mu,\nu) <\infty$, then there exists an optimal transference plan $\pi\in  \Pi(\mu,\nu)$ (i.e. a minimizer of \eqref{eq:OT}). 
Furthermore, any $\pi \in \Pi(\mu,\nu)$ is optimal if and only if there is a convex, lower semi-continuous, function $\psi:\R^n\to (-\infty,+\infty]$ such that
$\pi$ is concentrated in the graph of $\pa \psi$, that is
 \begin{equation}\label{eq:supgraph}
\supp(\pi) \subset \mathrm{Graph}(\pa\psi):= \left\{(x,z)\in  \RR^n\times \RR^n\; |\; z\in \pa\psi(x)\right\} \mbox{.}
 \end{equation}
\end{theorem}
Recall that for a convex function $\psi$, the subdifferential of $\psi$ at $x$ is given by
$$
\pa\psi(x) := \left\{ z\in \RR^n\, ;\, \psi(y)\geq \psi(x)+z\cdot(y-x) \mbox{ for all } y\in\RR^n\right\} .
$$
We will also use the notation $\pa\psi(A) = \cup_{x\in A} \pa\psi(x)$.

In this general framework, neither the optimal plan $\pi$ nor the {\it Kantorovich potential} $\psi$ are unique.
We also note that without changing the values of $\psi$ in $\supp \mu$, we can assume (see Appendix \ref{app:rest}) that $\psi(x) = \sup_{z\in  \O } ( x\cdot z - \phi(z) )$ for some function $\phi:\O \to (-\infty,+\infty]$. Since $\O$ is convex, we can thus assume that 
\begin{equation}\label{eq:O}
\pa \psi(\R^n) \subset \overline \O,
\end{equation}
and in fact we can even assume that $\pa\psi(\R^n) $ is a subset of the closed convex envelop of $\supp \nu$.
Since $\O$ is bounded, this implies that $\psi$ is Lipschitz continuous in $\R^n$.

There is a vast literature devoted to the properties of Kantorovich potentials when the measures $\mu$ and $ \nu$ 
are of the form 
\begin{equation}\label{eq:assfg1}
\mu = f(x)\, dx \quad \mbox{ and }  \quad \nu=g(y)\, dy
\end{equation}
for some measurable functions $f$ and $g$ satisfying
\begin{equation}\label{eq:assfg2}
\lambda \leq f(x) \leq \Lambda \qquad \forall x\in \Omega, \quad  \lambda \leq g(y) \leq \Lambda \qquad \forall y\in \O.
\end{equation}
Before stating our assumptions and main results, we will briefly review (in a completely non-exhaustive way) some key aspects of the classical regularity theory in that framework.
\subsection{Regularity theory under assumptions \eqref{eq:assfg1}-\eqref{eq:assfg2}}\label{sec:intro2}
Under assumptions \eqref{eq:assfg1}-\eqref{eq:assfg2}, the Kantorovich potential is uniquely defined (up to a constant) on $\Omega$. In fact, we have:
\begin{theorem}[Brenier's theorem \cite{brenier1991polar}]\label{thm:brenier}
If $\mu$ does no give mass to small sets, then there is a unique optimal transference plan $\pi$, given by 
$d\pi(x,y) =d\mu(x) \delta(y=\na\psi(x)$ where $\na \psi$ is the unique ($\mu$-a.e.) gradient of a convex function satisfying $\na \psi \# \mu = \nu$.
Furthermore, we have 
$$\supp(\nu) = \overline {\na \psi(\supp(\mu)))}.$$
\end{theorem}
In this framework, the so-called {\it Brenier map} $T=\na \psi$ is the unique solution of the Monge problem and when 
the measures $\mu$ and $\nu$ satisfy \eqref{eq:assfg1}-\eqref{eq:assfg2}, the function $\psi$ solves  the generalized Monge-Amp\`ere equation 
\begin{equation}\label{eq:MA0}
\det ( D^2 \psi (x)) = h(x) \qquad \mbox{ with } h(x):= \frac{\mu(x)}{\nu(\na \psi(x))} \qquad \mbox{ in } \Omega
\end{equation}
with second type boundary conditions \eqref{eq:O}.
In particular, $\psi$ satisfies
\begin{equation}
\label{eq:pureMongeAmpere}
 \frac{\lambda}{\Lambda} \chi_\Omega \leq \det D^2\psi \leq \frac{\Lambda}{\lambda} \chi_\Omega
 \end{equation}
in a weak sense (the {\it Brenier sense}) -  see for instance  \cite{caffarelli1992regularity, philippis2013regularity}.
When $\O$ is convex, 
Caffarelli proved in \cite{caffarelli1992regularity} that $\psi$ is a strictly convex  solution of \eqref{eq:pureMongeAmpere}  in the following {\it Alexandrov sense}:
\begin{equation}
\label{eq:MAalex}  \frac{\lambda}{\Lambda} | A\cap \Omega| \leq |\pa\psi(A)| \leq  \frac{\Lambda}{\lambda}  |A\cap\Omega| \qquad \mbox{for any Borel set $A\subset \RR^n$.}
\end{equation}
The regularity theory for the classical Monge-Amp\`ere equation, developed by Caffarelli \cite{caffarelli1990localization,caffarelli1991some,caffarelli1992boundary,caffarelli1992regularity,caffarelli1996boundary}, for strictly convex solutions of \eqref{eq:MAalex} then  implies that $\psi$ is $C^{1,\alpha}_{loc}$ (see also \cite{FM}).
Further regularity ($C^{2,\alpha}_{loc}(\Omega)$) for the solution of Monge-Amp\`ere equation \eqref{eq:MA0} with Dirichlet boundary conditions 
where proved for example by Caffarelli \cite{Caffarelli-interior} when the right hand side $h$ is $C^\alpha_{loc}(\Omega)$.

\medskip

Since further H\"older regularity is not expected under condition  \eqref{eq:assfg2} alone, it is natural to ask about the Sobolev regularity of $\psi$. 
In the framework of Monge-Amp\`ere equation \eqref{eq:MA0} (with Dirichlet boundary condition),  
Caffarelli \cite{Caffarelli-interior} proved that the solution is in $W^{2,p}$ provided the  right-hand side satisfies $\| h-1\|_{L^\infty}\leq \eta$ for some small enough $\eta$ (depending on $p$).
When $f$ and $g$ only satisfy condition  \eqref{eq:assfg2}, De Philippis and Figalli \cite{DF} proved that $D^2 \psi \in  L \log L_{loc}$. This result was improved by   De Philippis, Figalli and Savin \cite{DFS} who proved that $ \psi \in W^{2,1+\eps}_{loc}$ for some small $\eps>0$.

\subsection{General measures and formal presentation of our results}\label{sec:intro3}
In this paper, we do not make either assumption~\eqref{eq:assfg1} or~\eqref{eq:assfg2}, instead requiring only that $\mu$ and $\nu$ are comparable to the Lebesgue measure when restricted to balls of radius greater than or equal to some fixed $\delta>0$ (see Assumption \ref{ass:1}).
A simple and very important example we have in mind is that of measures which are sums of Dirac masses  approximating some absolutely continuous measures satisfying \eqref{eq:assfg1}-\eqref{eq:assfg2}.
Our goal is then to prove the equivalent of the $C^{1,\alpha}_{loc}$, $W^{2,1}_{loc}$ and $W^{2,1+\eps}_{loc}$ regularity results mentioned above, in our framework.
Since we consider measures that may not be absolutely continuous with respect to the Lebesgue measure, the Kantorovich potential $\psi$ (which still exists but may not be unique) does not solve the Monge-Amp\`ere equation (either \eqref{eq:pureMongeAmpere} or \eqref{eq:MAalex}) and none of the results discussed above apply.
In fact,  Kantorovich potentials are not expected to be either strictly convex or $C^1$: Even in the semi-discrete case (when $\nu$ is absolutely continuous but $\mu$ is a sum of Dirac masses), $\psi$ is a piecewise affine function
 (see \cite{Berman} and Appendix \ref{app:rest}). In particular $\psi$ has both
 `flat parts' and `corners'.
To our knowledge, few quantitative estimates on the convex function $\psi$ are known in this setting. 

 This discrete framework with $\mu$, $\nu$ sums of Dirac masses is important for numerical computations of   optimal maps between regular measures, as described in Appendix \ref{app:app}. There exists an extensive literature investigating numerical schemes for optimal transportation. This is a challenging question, both from the point of view of the strongly non-linear Monge-Amp\`ere equation or when starting from the Monge-Kantorovich formulation. In the later case for example, it would be natural to discretize the starting and target measures by sums of Dirac masses (which fits well with our framework) since that formally reduces optimal transportation to a linear programming problem. But this linear programming behave quadratically in the number of points in the discretization, leading to highly non-trivial computational challenges. We briefly refer to \cite{BCC,BD}, \cite{Levy}, \cite{MT,Merigot} or \cite{oberman2015efficient} for examples of various numerical schemes.  Several recent works have focused on proving convergence rates for these numerical schemes: see for instance \cite{Berman} for the semi-discrete case and \cite{li-nochetto} for the fully discrete case, together with \cite{KMT, merigot2019}.

Such convergence rates are not the goal of this paper. Instead,  
we develop  a regularity theory for the potential $\psi$, "up to a length scale depending on $\delta$", which parallels the classical regularity theory developed under assumptions \eqref{eq:assfg1}-\eqref{eq:assfg2}. To roughly summarize the main results stated in the next section, if we consider the optimal transportation problem corresponding to some potential $\psi$ between measures $\mu$ and $\nu$ that are discretized at some scale $\delta$, then we will prove  that $ \bar K_r\star \Delta\psi$ belongs to $L^p$ for some $p>1$ where  $\bar K_r$ is a  mollifying kernel at scale $r$ with $\delta^{1/2}\lesssim r$. We expect that the precise uniform regularity results that we derive here will prove useful for the analysis of complex numerical schemes;  and in particular multi-scale schemes.

 We want to emphasize that an important aspect of  our results is to obtain the regularity of $\bar K_r\star \Delta \psi$ for the conjectured critical scale $r\sim \delta^{1/2}$.
Indeed, this rate is consistent with the known rates of convergence of numerical schemes:
If  $\mu_\delta$ and $\nu_\delta$ are discretizations at scale $\delta$ of some nice measures $\mu_0$ and  $\nu_0$, then 
the corresponding Kantorovich potentials $\psi_\delta$ and $\psi_0$ are expected to satisfy
$\psi_\delta-\psi_0=O(\delta^{1/2})$ in some appropriate norms - see  \cite{li-nochetto} where weighted $L^2$ error estimates  on the associated transport maps are derived (see also \cite{Berman} in the semi-discrete case).
Using the regularity of $\psi_0$, it is possible, via some interpolation, to derive some $L^p$ bounds on $\bar K_r\star \Delta\psi_\delta$. However such an approach will lead to regularity results at scale $r\geq   \delta^{\alpha}$ with $\alpha <1/2$.
In addition, we note that the convergence results of \cite{Berman,li-nochetto} requires some $C^2$ regularity assumptions for the limiting potential $\psi_0$ while our result is completely independent of the properties of $\psi_0$ (in fact it does not even require the existence of $\psi_0$).

The results of the present paper should also be compared with the previous work of the authors (with M. Molina) in 
 \cite{JMM}: 
Under similar assumptions on the measures $\mu$ and $\nu$, in dimension $2$,  the function $\psi$ is proved to be strictly convex up to some discrete scale and approximately $C^1$ again up to some discrete scale. 
The $C^{1,\alpha}$ and $W^{2,p}$ regularity results proved here (in any dimension) are thus a significant improvement of these results. 
On the other hand, we note that the results in~\cite{JMM} are local and do not require any control 
on $\psi$ near the boundary of $\Omega$, while in the present article we require some sort of boundary conditions (see a precise discussion in the next section). 
This is in complete accord with what is known for solutions of Monge-Amp\`ere equation with a right-hand side bounded from below and from above: strict convexity holds locally in dimension~$2$ independently of the boundary condition but
 in higher dimensions, classical counter-examples show that boundary conditions are necessary. 
  
 Finally, we  mention the question of partial regularity to optimal transportation between $C^{0,\alpha}$ densities which shows that the map $T=\nabla \psi$ is $C^{1,\alpha}$ up to removing sets of measure $0$; see~\cite{DF2,FigalliKim10} and \cite{Goldman17} for a variational approach. This suggests that some notion of approximate $C^{2,\alpha}$ estimate on $\psi$ (and not only $C^{1,\alpha}$ as we prove here) may hold, though with stronger assumptions on the measures $\mu$ and $\nu$ than what we are making here. 

\section{Main assumptions and main results}
\subsection{The measures $\mu$ and $\nu$}
In this paper, we assume that the measures $\mu$ and $\nu$ are comparable to the Lebesgue measure up to a certain scale. Our main assumption can be written as follows:
\begin{assumption}\label{ass:1}
There exist $\lambda$, $\Lambda>0$ such that
\begin{equation}\label{eq:munu1}
\lambda |B_r| \leq \mu(B_r) \leq\Lambda |B_r|, \quad \mbox{for all ball $B_r\subset \Omega$ with $r \geq  \delta$} 
\end{equation}
and 
\begin{equation}\label{eq:munu2}
\lambda |B_r| \leq \nu(B_r) \leq\Lambda |B_r| \quad \mbox{for all ball $B_r\subset \O $ with $r \geq \delta$.} 
\end{equation}
\end{assumption}
We note that \eqref{eq:munu1} is equivalent to assuming that the measure convolution $\mu_\delta = \mu\star \frac{1}{|B_\delta|}\chi_{B_\delta}$ is absolutely continuous with respect to the Lebesgue measure and has density $f_\delta$ satisfying
$$ \lambda \leq f_\delta \leq \Lambda  \quad\mbox{ a.e. in } \Omega_\delta = \{x\in\Omega \,;\, d(x,\pa\Omega)>\delta\}.$$
We also point out that the use of balls in \eqref{eq:munu1}-\eqref{eq:munu2} is not crucial.
Indeed, up to multiplying $\delta$ by a constant depending only on the dimension, we could formulate our assumption with ellipsoids whose smallest radius are greater than $\delta$ or even with general convex sets with inner radius greater than $\delta$
(see Proposition \ref{prop:bhkh})

Assumption \ref{ass:1} is natural in situation in which the measures $\mu$ and $\nu$ are obtained by approximating measures satisfying \eqref{eq:assfg1}-\eqref{eq:assfg2}. This construction, which plays a key role in the numerical computation of optimal maps (see \cite{Berman,li-nochetto}) is recalled in Appendix \ref{app:app}.

We already mentioned that in this setting, it does not make sense to write the Monge-Amp\`ere equation and that many of the classical tools (e.g. comparison principle) are thus no longer available to prove our result.
Another important observation, which is less fundamental and more technical, is the fact that the renormalization of convex set provided by John's Lemma \ref{lem:J} is extensively used in the literature to prove regularity results. Unfortunately, Assumption \ref{ass:1} is not invariant under affine transformations (since the image of a ball could be a very elongated ellipsoid).
We will thus need to largely forgo the use of such renormalization in this paper.

\subsection{The potential $\psi$}
Throughout the paper, $\psi:\R^n\to \R$ is a convex lower semi-continuous function associated to an optimal transference plan $\pi$ as in  Theorem \ref{thm:OT1} and satisfying \eqref{eq:O}.
In particular, $\psi$ satisfies
\begin{equation} \label{eq:ineq1}
\mu(A) \leq  \nu (\pa\psi(A))\quad \mbox{ for all Borel sets $A\subset \Omega$}
\end{equation}
where $\pa\psi(A) = \cup_{x\in A} \pa\psi(x)$.
Indeed,  we can write
$$ 
\mu(A) = \int_A \int_{\RR^n} d\pi =  \int_A \int_{\pa \psi(A)} d\pi \leq  \int_{\RR^n} \int_{\pa \psi(A)} d\pi  = \nu(\pa\psi(A)).
$$
This inequality might be strict in general but that we also have the inequality
\begin{equation}\label{eq:ineq2}
 \mu(\pa\psi^*(B)) \geq \nu(B)\quad \mbox{ for all Borel sets $B\subset \O$}
\end{equation}
where $\psi^*$ denotes the Legendre transform, defined by
$$\psi^*(z) = \sup_{x\in\RR^n} \big( x\cdot z -\psi(x)\big) \qquad z\in \R^n.$$
Indeed, we recall that  $\psi^*$ is also convex and lower semi-continuous (and \eqref{eq:O} implies that $\psi^*(z)=+\infty$ when $z\notin \overline \O$)
and that we have the following equivalences:
\begin{equation}\label{eq:subeq}
x\cdot y = \psi(x)+ \psi^*(y)\; \Longleftrightarrow \;y \in\pa \psi(x) \; \Longleftrightarrow \; x \in\pa \psi^*(y). 
\end{equation}
In particular  \eqref{eq:supgraph} is equivalent to
$ \supp(\pi) \subset \left\{(x,z)\in  \RR^n\times \RR^n\; |\; x\in \pa\psi^*(z)\right\} $ and we can prove 
\eqref{eq:ineq2} just as we proved \eqref{eq:ineq1}.

\medskip

These two properties \eqref{eq:ineq1} and \eqref{eq:ineq2}, together with the boundary condition \eqref{eq:O}, are actually the only properties of $\psi$ that we will use to prove our result.

The relation between these inequalities and Theorem \ref{thm:OT1} was made clear by a result that we proved in \cite{JMM} and which we recall here for the reader's convenience:
\begin{theorem}\label{thm:optimal}
Assume   $\mu,\;\nu$ are  two probability measures on $\RR^n$ such that $C(\mu,\nu)<\infty$
and let $\psi$ be a proper lower semi-continuous  convex  function satisfying
\eqref{eq:ineq1} and \eqref{eq:ineq2}.
Then there exists an optimal transference plan $\pi\in \Pi(\mu,\nu)$ (minimizer of \eqref{eq:OT}) such that  $\mathrm{supp}(\pi) \subset \mathrm{Graph}(\pa\psi)$.
\end{theorem}

\subsection{Sections of convex functions}
As mentioned above, a fundamental difference between our work and most previous work that deals with absolutely continuous measure $\mu$ and $\nu$, is the fact that we cannot use Monge-Amp\`ere equation and the comparison principle usually associated to elliptic equations and which plays a key role in 
the work of Caffarelli as well as in the more recent partial regularity theory of
 \cite{Yu07,Figalli10,FigalliKim10}.
 Note that  a variational approach (relying on optimal transportation arguments rather than using some barriers for Monge-Amp\`ere equation) to the partial regularity theory of \cite{FigalliKim10}  was recently developed in \cite{Goldman17, Goldman18}.

On the other hand, we will rely heavily on another tool that has proved extremely useful in the study of Monge-Amp\`ere equation:  The sections of convex functions. We recall that 
given $x_0\in\Omega$, $p_0\in \pa\psi(x_0)$ and $t\geq 0$,  the section centered at $x_0$ with height $t$ is the convex set:
$$ S(x_0,p_0,t) := \{ x\in\R^n\, ;\, \psi(x) \leq \psi(x_0)+ p_0\cdot(x-x_0)+t\}. $$

Sections associated to solutions of Monge-Amp\`ere equation have good properties and play a key role in most regularity results, see for instance  \cite{Caffarelli91},
\cite{Gutierrez00},
\cite[Chapter 3]{Gutierrez01}.
An important contribution of this paper, and a key aspect of our proofs, will be to show that these classical properties still hold  under Assumption \ref{ass:1}, provided $t\geq C_0\delta$ for some  constant $C_0$ depending only on the dimension $n$ and on the diameters of $\Omega$ and $\O$. 
These properties are summarized in  Proposition \ref{prop:sec} and  proved in Section \ref{sec:sec}. In  order to avoid the use of renormalization of convex sets (which, as noted earlier, is not compatible with Assumption \ref{ass:1}), we will rely on the notion of polar  body of a convex set to prove several of our results.
\medskip

A final remark is that our regularity results will be interior regularity results: they hold in a subset $\Omega'\subset\subset \Omega$ such that
there exists $\rho$  such that
\begin{equation}\label{eq:rho1} 
S(x,p,\rho)\subset \Omega \quad\forall x\in \Omega' \quad\mbox{for some } p\in \pa\psi(x).
\end{equation}
The existence of such a $\rho>0$ for a given subset $\Omega'$
is classical when the measures satisfy \eqref{eq:assfg1}-\eqref{eq:assfg2} (and $\O$ is convex), but it
 is not obvious in more general framework.
 We will show in Appendix \ref{app:app} that when $\mu$ and $\nu$ are obtained as discrete approximation of regular measures, then $\rho$ satisfying \eqref{eq:rho1} exists provided $\delta$ is small enough for any $\Omega'\subset\subset \Omega$. Technically speaking~\eqref{eq:rho1} then holds for all $p\in \pa\psi(x)$ though we do not need this stronger version for our analysis.
 
 The fact that some sort of condition is needed near the boundary of $\Omega$ is however well-understood. In the case $\delta=0$, the regularity of the solutions to Monge-Amp\`ere equation depends on the boundary conditions. For example in dimension~$3$ or more, local solutions to Monge-Amp\`ere equation with a right-hand side bounded from below and from above may fail to be strictly convex (dimension~$n=2$ is a special case for which we refer in particular to the discussion in~\cite{JMM}). Because of our setting with discrete measure, straightforward boundary conditions would not make much sense and are replaced by~\eqref{eq:rho1} which is a natural extension: To violate~\eqref{eq:rho1}, one would in particular need to have very elongated sections for small $\rho$ that touch the boundary of $\Omega$.
\subsection{Main results}
Our first result concerns the $C^{1,\alpha}_{loc}$ regularity of $\psi$ up to some length scale depending on $\delta$. 
Forzani and Maldonado \cite{FM} proved that the engulfing property of sections (see Proposition \ref{prop:sec}-(iii)) is  equivalent to the $C^{1,\alpha}$ regularity of the potential. 
Using similar ideas, we will show the $C^{1,\alpha}$ regularity up to a scale depending on $\delta$.
In the theorem below, the constants $C_0$, $\theta$ and  $\beta$  are the constants appearing in Proposition \ref{prop:sec} and in \eqref{eq:secd}:
 \begin{theorem}[$C^{1,\alpha}$ regularity]\label{thm:alpha}
Let  $\psi:\R^n\to \R$ be a convex function satisfying \eqref{eq:O}, 
\eqref{eq:ineq1} and \eqref{eq:ineq2} with $\mu$ and $\nu$ satisfying \eqref{eq:munu1} and \eqref{eq:munu2} for some $\delta>0$.
Let $\Omega'\subset\subset \Omega$ be such that
there exists $\rho> C_0\delta$   such that \eqref{eq:rho1}  holds.
For all $\Omega'' \subset\subset \Omega'$ there exists
a constant  
$C$ depending on 
$\mathrm{dist}(\pa\Omega',\pa\Omega'')$, $\rho$, $ \pa\psi(\Omega)$, $\Omega$, $\lambda$, $\Lambda$ and $n$
such that
$$
0\leq  \psi(x) -\psi(y) -q\cdot (x-y) \leq \overline C |x-y|^{1+\frac 1 \theta}  \qquad \forall q\in \pa\psi(y)
$$
for all $x,y\in \Omega''$  satisfying
$$ |x-y|\geq C\delta^{1/\beta}.$$
\end{theorem}
\medskip

Next, we turn our attention to the Sobolev regularity of $\psi$. When $\delta=0$ (that is when \eqref{eq:assfg1}-\eqref{eq:assfg2} holds), 
De Philippis and Figalli \cite{DF} proved the following result:
\begin{theorem}[\cite{DF}]\label{thm:DF}
Let $\Omega\subset \R^n$ be bounded convex domain and let $u:\overline \Omega \to\R$ be an Alexandrov solution of 
\begin{equation}\label{eq:MA}
\begin{cases}
\det D^2 u =h & \mbox{ in }\Omega\\
u=0 & \mbox{ on } \pa\Omega
\end{cases}
\end{equation}
with $0<\lambda\leq h(x)\leq \Lambda $ for all $x\in \Omega$. Then for any $\Omega'\subset\subset\Omega$ and $k\in \NN$, there exists a constant 
$C$ depending on $k$, $n$, $\lambda$, $\Lambda$, $\Omega$ and $\Omega'$ such that
\begin{equation}\label{eq:DF}
\int_{\Omega'}\|D^2 u \|\log^k(2+\|D^2 u\|)\, dx \leq C.
\end{equation}
\end{theorem}
The function $u$ in Theorem \ref{thm:DF} satisfies
\begin{equation}\label{eq:alex} 
\lambda |B| \leq |\pa u(B)| \leq \Lambda |B| \qquad \mbox{ for all Borel set $B\subset \Omega$.} 
\end{equation}
These conditions are also satisfied by our function $\psi$ when  Assumption \ref{ass:1} holds  with $\delta=0$ if we assume that  $|\pa \O|=0$ and that $\psi$ is $C^1$ and strictly convex.
Indeed, when $\mu$ and $\nu$ satisfy \eqref{eq:munu1}-\eqref{eq:munu2} with $\delta=0$, 
then \eqref{eq:ineq1}, \eqref{eq:ineq2} imply
$$  \lambda |A| \leq |\pa\psi(A)| \qquad\forall A\subset \Omega,\qquad  |B| \leq \Lambda |\pa\psi^*(B)| \qquad\forall B\subset \O .$$
Since  $\pa\psi(A)\subset \overline{\O }$ for all $A\subset \Omega$ and if $|\pa\O |=0$,  the second inequality gives
$$ |\pa \psi(A)| \leq \Lambda |\pa\psi^*(\pa\psi (A))|.$$
When $\psi$ is strictly convex, we have $\pa\psi^*(\pa\psi (A))=A$ so we recover \eqref{eq:alex}.

Conditions \eqref{eq:munu1}-\eqref{eq:munu2} (with $\delta=0$) are actually enough to carry out an adaptation of the arguments of \cite{DF} and derive \eqref{eq:DF}. In fact, we prove:
\begin{theorem}[$W^{2,1}_{loc}$ regularity when $\delta=0$]\label{thm:1}
Let $\psi:\R^n\to \R$ be a convex function satisfying \eqref{eq:O},
\eqref{eq:ineq1}, \eqref{eq:ineq2}, with $\mu$ and $\nu$ satisfying \eqref{eq:munu1} and \eqref{eq:munu2} with $\delta=0$. Assume moreover that $\Delta\psi\in L^1(\Omega)$ and let
$\Omega'\subset\subset \Omega$ such that there exists $\rho>0$    such that \eqref{eq:rho1} holds.
There exists a constant $C$ depending on $\Omega$, $\Omega'$, $\O $, $n$, $\lambda$, $\Lambda$ such that
\begin{equation}\label{eq:lapdel} \int_{\Omega'} \Delta \psi \log(2+\Delta \psi) \, dx \leq C  .
\end{equation}
\end{theorem}
Since $\psi$ is a convex function, we have $\frac 1 n \Delta \psi\leq  \| D^2 \psi\| \leq \Delta \psi$ and so
\eqref{eq:lapdel} is equivalent to the result of \cite{DF}. We will nevertheless give a detailed proof of Theorem \ref{thm:1}. The overall arguments are the same as in \cite{DF}, although some   minor details are different (for example we avoid the use of the maximal function and renormalization of convex sets). 
Since Theorem \ref{thm:1} can also be obtained by taking the limit $r\to 0$ in the next result, this proof is not really necessary, but it allows the reader to see the main ideas of our proof in a simpler framework.

The main goal of the paper is to derive Sobolev regularity results when $\delta>0$.
Since  the potential $\psi$ might be piecewise affine,  the measure $\Delta \psi$ may be singular
and we cannot expect  \eqref{eq:DF} to hold in that case. 
Instead, we will derive a similar bound for an approximation of the Laplacian   which measures the mass of the second derivative up to a scale $r$.
More precisely, we fix a radially symmetric kernel $K:\R^n \to [0,\infty]$, supported in $B_1$ and such that 
$$\int K(x)\, dx = 1, \qquad \int_{B_1} |x|^2 K(x) \, dx>0.$$ 
We then define for any $r>0$
$$ 
\Delta_r\psi (x):= \frac{m_0}{r^2}   \int K_r(y) [\psi(x+y) -\psi(x)]\, dy, \qquad K_r(x) = \frac 1 {r^n} K\left(\frac x r\right)
$$
with $m_0=\frac{2}{ \int_{\R^n} y_1^2 K(y) \, dy} $.  
We have in particular that $\Delta_r\psi (x)\geq 0$ as long as $\psi$ is convex  and we will see that $\lim_{r\to 0} \Delta_r\psi = \Delta \psi$ when $\psi$ is $C^2$. 
We shall derive estimates on $\Delta_r\psi$ for any choice of $K$ satisfying the assumptions listed above. Some natural  choices  of $K$ include $K = \frac{1}{|B_1|}\chi_{B_1}$ for which we have 
$$
\Delta_r\psi (x):= \frac{1}{r^2}   \left[\fint_{B_r(x)}  \psi(y)\, dy -\psi(x)\right], 
$$
and $K=  \frac {1}{\mathcal H^{n-1}(\pa B_1)} \mathcal H^{n-1} |_{\pa B_1}$ which leads to
$$
\Delta_r\psi (x):= \frac{1}{r^2}   \left[\fint_{\pa B_r(x)}  \psi(y)\, dy -\psi(x)\right].
$$
With these notations, the  second main result of this paper is the following theorem:
\begin{theorem}[$W^{2,1}_{loc}$ regularity when $\delta>0$]\label{thm:2}
Let  $\psi:\R^n\to \R$ be a convex function satisfying 
\eqref{eq:ineq1}, \eqref{eq:ineq2}, with $\mu$ and $\nu$ satisfying \eqref{eq:munu1} and \eqref{eq:munu2} with $\delta>0$.
Let $\Omega'\subset\subset \Omega$ be such that
there exists $\rho>0$   such that \eqref{eq:rho1} holds.
There exist a constant $C$ depending on $\Omega$, $\rho$, $\O $, $n$, $\lambda$, $\Lambda$  such that
$$ \int_{\Omega'} \Delta_r \psi \log(2+\Delta_r\psi) \, dx \leq C \qquad \forall r > \delta^{\frac 1 2}.$$
\end{theorem}
Since the constant $C$ does not depend on $r$, we see that when $\delta=0$, we can take the limit $r\to 0$ and recover the result of Theorem \ref{thm:1}.
 For that reason, we will not be very careful with the a priori regularity of $\psi$ in the proof of  Theorem \ref{thm:1}: A priori,  $\Delta \psi$ is a Radon measure and some of the computations in Section \ref{sec:cheb} should involve a decomposition of that measure into a absolutely continuous part and a singular part. For simplicity, we just assume that $\Delta\psi$ is locally integrable.
This is not an issue with the proof of Theorem \ref{thm:2} since $\Delta_r \psi$ is a Lipschitz function for all $r>0$. 
Our proof of Theorem \ref{thm:2} thus also provides a  proof of Theorem \ref{thm:1} which does not require any a priori knowledge on the structure of $\Delta \psi$.

 Following \cite{DFS}, we can then improve the argument to show: 
 \begin{theorem}[$W^{2,1+\eps}_{loc}$ regularity when $\delta>0$]\label{thm:sobolev}
Let $\psi:\R^n\to \R$ be a convex function satisfying 
\eqref{eq:ineq1}, \eqref{eq:ineq2}, with $\mu$ and $\nu$ satisfying \eqref{eq:munu1} and \eqref{eq:munu2}) with $\delta>0$.
Let $\Omega' \subset\subset \Omega$ be such that 
there exists $\rho>0$  such that \eqref{eq:rho1} holds.
For any $x_0\in \Omega'$ and $R$ such that $B(x_0,2R)\subset\Omega'$, there exists $p>1$,  $C$  (depending on 
$R$, $\Omega$, $\rho$, $n$, $\lambda$, $\Lambda$, $L(\pa\psi(\Omega))$)
such that 
\[
\int_{ B(x_0,R)} \, (\Delta_r \psi)^p  \, dx \leq C \qquad\forall r > \delta^{\frac 1 2}.
\]
\end{theorem}

We point out that we can also write $\Delta_r\psi (x)$ as $\Delta (G_r \star \psi)$ or $\bar K_r \star \Delta \psi$. 
So that our results can be interpreted as a priori Sobolev norm on a regularization $ G_r \star \psi$ of $\psi$, instead of a $L^p$ bound on the discrete Laplacian $\Delta_r$ of $\psi$.
A simple computation leads to the relation $K_r = \delta +\Delta G_r $ and $G_r = \Gamma - \Gamma \star K_r$ (where $\Gamma$ denotes the fundamental solution $- \Delta \Gamma =\delta$).
For example we can check that the kernel $K=  \mathcal H^{n-1} |_{\pa B_1}$ corresponds to
$G (x)= c_n \left(\frac{1}{|x|^{n-2}}-1\right)_+$ for some normalization constant $c_n$.

\subsection{What is new in our method}
We emphasize that the results stated above rely on some key improvements with respect to the existing strategies:
\begin{itemize}
\item Precise properties of sections. A key point both for the $C^{1,\alpha}$ and $W^{2,p}$ regularity in the continuous case is the so-called engulfing property  of sections $S(x,p,t)$ (see Proposition~\ref{prop:sec} later for a precise formulation). This in particular allows to use Vitali's covering lemma on sections. Unfortunately, the engulfing property does not hold for sections of a generic convex function and instead relied on the function $\psi$ solving the Monge-Amp\`ere equation. An important technical contribution of this paper is thus to extend the engulfing property to our discrete setting, based on optimal transportation instead of Monge-Amp\`ere and {\em up to the precise critical discrete scale} $t\sim \delta$.
\item ``Fattening'' small sets. The classical proof of $W^{2,p}$ regularity in the case $\delta=0$ relies on the identification in each section $S(x,p,t)$ of a critical set $\Sigma$ over which $\Delta\psi$ is large enough and with $\Sigma=\pa\psi^*(V)$ for $|V|$ large enough as well. The proof of Theorem~\ref{thm:2} requires enlarging those sets $\Sigma$ and $V$ so that $\Delta_r \psi$ is large on $\Sigma$ (which requires $\Delta\psi$ to be large on a neighborhood of $\Sigma$) and $\Sigma$ and $V$ to be large enough to apply Assumption~\ref{ass:1}. The way to enlarge those sets is delicate if one wants to keep some notion of duality between them $\Sigma=\pa\psi^*(V)$ and be compatible with the critical scale $r\sim\delta^{1/2}$. 
\end{itemize}

\medskip

\noindent{\bf Outline of the paper:}
In the next section, we recall several facts about convex sets and sections of convex functions that will be used throughout the rest of the paper. The main result of that section is Proposition \ref{prop:sec}, which summarizes the properties of the sections of $\psi$. These properties are classical when $\delta=0$, and generalizing them to the case $\delta>0$ is an important contribution of this paper. The proof of Proposition \ref{prop:sec} is detailed in Section \ref{sec:sec}.
The proof of Theorem \ref{thm:1} is presented in Section \ref{sec:delta0}. It relies on ideas introduced in \cite{DF} but it has been slightly modified  to avoid the use of the maximal function and the use of renormalizations. We present it here mainly because it allows us to present the main ideas and estimates in the simpler framework $\delta=0$ before presenting the more complicated proof of Theorem \ref{thm:2} in Section \ref{sec:delta}. 
%
\section{Preliminaries}

\subsection{Notations}
The proofs will make extensive use of the notion of inner and outer radii of a convex set, defined as follows:
\begin{align*}
\ell(K) & := \mbox{ the radius of the largest ball contained in $K$ (inner radius)}\\
L(K) & :=  \mbox{ the radius of the smallest ball containing $K$ (outer radius).}
\end{align*}

Throughout the paper, $c$ and  $C$ denote some constants, which may change value from line to line, that depend (unless otherwise indicated) only on the dimension $n$, the sets $\Omega$, $\O $ and the constant $\lambda$, $\Lambda$.
When comparing quantities, we will use the notation $a\sim b$ to indicate that there are constant $c_1$, $c_2$ depending only on the dimension $n$ such that
$$ c_1 a \leq b \leq c_2a.$$
The norm of a matrix $A$ is defined by $ \|A\| = \sup_{|x|=1} |Ax|$.
Given a positive definite matrix $A$, we will use the notation
$$ | x| _A := \sqrt{x^*Ax} \qquad x\in \R^n.$$

\subsection{John's Lemma}
We recall the following classical lemma:
\begin{lemma}\label{lem:J} [John's Lemma]
Let $K$ be a bounded convex set with non-empty interior. There exists an affine transformation $T(x) = Nx+b$ such that
$$ B_1(0) \subset T(K)\subset B_n(0).$$
Up to a translation of $K$, we can always assume that $b=0$. We then have 
\begin{equation}\label{eq:KA} 
\{ x\in \R^n\, ;\, x^*Ax \leq 1\}\subset  K \subset \{ x\in \R^n\, ;\, x^*Ax \leq n^2\} \qquad \mbox{ with } A = N^*N .
\end{equation}
\end{lemma}

We denote by $\lambda_1\geq \cdots \geq \lambda_n>0$ the eigenvalues of the positive symmetric matrix $A=N^*N$. The set $x^*Ax =1$ is an ellipsoid with radii $\frac{1}{\sqrt {\lambda_i}}$, $i=1,\dots,n$. In particular,  
\eqref{eq:KA} implies that
the inner and outer  radii of $K$ satisfy $\ell(K)\sim \frac{1}{\sqrt {\lambda_1}}$ and $L(K)\sim \frac{1}{\sqrt {\lambda_n}}
$. More precisely, we have
$$ 
\frac{1}{\sqrt {\lambda_1}} \leq \ell(K) \leq  \frac{n}{\sqrt {\lambda_1}}\quad\mbox{and} \qquad  \frac {1}{\sqrt{\lambda_n}}\leq L(K) \leq \frac {n}{\sqrt{\lambda_n}}.$$
We also have $\|N^*\| \|N \| = \| A\| \sim \mathrm{Tr} (A)$  (the first equality can be proved using the polar decomposition of matrices) and 
\begin{equation}\label{eq:Kl}
 \frac{1}{\ell(K)^2} \leq  \mathrm{Tr} (A)\leq  \frac{n^3}{\ell(K)^2}.
\end{equation}

\subsection{Revisiting Assumption~\ref{ass:1}}
We stated Assumption~\ref{ass:1} for balls where it is easily understood. We will however need to compare $\mu(A)$ or $\nu(A)$ with $|A|$ for more general sets, including some convex sets and sections. The following proposition
allows us to use the bounds \eqref{eq:munu1}-\eqref{eq:munu2} in such cases.
\begin{proposition}\label{prop:delta}
Let $\mu$ and $\nu$ satisfy Assumption \ref{ass:1}. 
Then there exists $\lambda'$ and $\Lambda'$ and $\beta$ such that
\begin{equation}\label{eq:muAd}
\lambda' |A_\delta| \leq \mu(A_\delta),\; \nu(A_\delta) \leq\Lambda' |A_\delta|, 
\end{equation}
for any Borel set $A$ and with $A_\delta = \cup_{x \in A}(S_\delta(x))\subset \Omega $ where $\{S_\delta(x)\}_{x\in A}$ is a family of convex sets satisfying $x\in S_\delta$, $\ell (S_\delta)\geq \beta \delta$, and for which Vitali's covering lemma (with constant $C_*$)  can be applied.
Furthermore, the constant $\lambda'$, $\Lambda'$ and $\beta$ only depend on $\lambda$, $\Lambda$ and $C_*$.
In particular, \eqref{eq:muAd} holds if $A$ is a convex set satisfying $\ell(A)\geq \beta\delta$.
\end{proposition}
The proof of this proposition can be found in Appendix \ref{app:A}.
%
\subsection{Sections of a convex function}
The notion of sections of a convex function $\psi:\R^n \to \R$ plays a central role in many regularity results of the theory of  optimal transportation, including in the approach that we present here.
Given $x_0\in\Omega$, $p_0\in \pa\psi(x_0)$ and $t\geq 0$, we recall that the section centered at $x_0$ with height $t$ is defined as the convex set:
$$ S(x_0,p_0,t) := \{ x\in\R^n\, ;\, \psi(x) \leq \psi(x_0)+ p_0\cdot(x-x_0)+t\}. $$
When $ \pa\psi(x_0)$ reduces to one point $\{\na \psi(x_0)\}$, we will simply write
$ S(x_0,t)$ for $ S(x_0,\na \psi(x_0),t)$. Furthermore, given $\gamma>0$ we denote by $\gamma S(x_0,p_0,t)$ the dilation of $S(x_0,p_0,t)$ with respect to the center of mass $x^*_0$\footnote{We could also use dilations with respect to $x_0$, as in \cite{DF}. The properties of section listed in Proposition \ref{prop:sec} are valid with dilation with respect to any point in $S$.}:
$$ \gamma S(x_0,p_0,t) : = \left\{ x_0^* + \gamma (x-x_0^*)\, ;\, x\in S(x_0,p_0,t)\right\}.
$$

The following simple result will be useful to control the inner radius of sections:
\begin{proposition}\label{prop:lL}
Let $\Omega$ be a bounded subset of $\R^n$ and $\psi:\Omega \to \R$ be a convex function such that   $\pa \psi (\Omega)$ is also a bounded subset of $\R^n$.
For any section $S=S(x_0,p_0,t)\subset \Omega$, we have 
$$ \ell(S) \geq \frac{t}{ L(\pa\psi(\Omega))},  \qquad  \ell(\pa\psi(S)) \geq \frac{t}{L(\Omega)}.
$$
\end{proposition}

The scaling of these bounds is not optimal in the framework of optimal transportation,
but is the best one can hope for in the case of general convex functions; in particular the first inequality is an equality  when $\psi(x)=L |x|$.

\begin{proof}
After replacing the function $\psi$ with  $\overline \psi(x) = \psi(x_0+x)-p_0\cdot x$, we can assume that $x_0=0$ and $p_0=0$. We then have $|\na \overline \psi(x)| = |\na \psi(x_0+x) -p_0|\leq L(\pa\psi(\Omega))$ a.e. $x\in S$.
This implies in particular that $\bar\psi (x) \leq \bar\psi(0)+ L(\pa\psi(\Omega)) |x|$ and so $B_{t/L(\pa\psi(\Omega))} (0)\subset S(0,0,t)$. The first inequality follows.

To prove the second inequality, we note that if $q\in \R^n$ is such that $|q|< \frac{t}{L(\Omega)}$, then the function $u(x)=\overline \psi(x) -\overline \psi(0)- q\cdot x$ satisfies $u(0)=0$ and $u(x) = t-q\cdot x\geq t - |q||L(\Omega)|>0$ for $x\in \pa S$. It follows that $u$ has a minimum in the interior of $S$. At such a point we have $0\in\pa u(x) $ and so $q\in \pa\overline \psi(x)\subset \pa\overline \psi(S)$.
We deduce that $B_{ \frac{t}{L(\Omega)}} (0)\subset \pa\bar  \psi(S)$ which implies the second inequality
\end{proof}

We   also recall the classical Alexandrov   estimate  (see for example \cite{Gutierrez00}):
\begin{proposition}\label{prop:alek}
Let $\psi:\Omega \to \R$ be a convex function
and assume that $0\in\Omega$, $0\in \pa\psi(0)$ and $t\geq 0$ is such that $S:=S(0,0,t)\subset \Omega$. Then
$$ |\psi(y)-t|^n \leq c_n L(S)^{n-1} \mathrm{dist}(y,\pa S) |\pa\psi(S)|$$
for some constant $c_n$ depending only on the dimension.
\end{proposition}

Finally, we summarize  in the following proposition the key properties 
 satisfied by the  sections of convex potentials appearing in the context of optimal transportation (that is under our assumptions  \eqref{eq:ineq1}-\eqref{eq:ineq2}  and \eqref{eq:munu1}-\eqref{eq:munu2}).

\begin{proposition}\label{prop:sec}
Let $\psi:\R^n\to\R$ be a convex function satisfying  \eqref{eq:ineq1}-\eqref{eq:ineq2} for measures $\mu$ and $\nu$ satisfying Assumption \ref{ass:1}.
There exists a constant  $C_0$   depending only on $n$, $L(\Omega)$  and $L(\pa\psi(\Omega))$  
such that when 
$$ t\geq C_0 \delta$$
we have
\item{(i)} There exists $C = C(n,\lambda,\Lambda)$ such that 
$C^{-1} t^n \leq  |S(x_0,p_0,t)|^2 \leq  Ct^n$ for all sections  $S(x_0,p_0,t)\subset \Omega$.
\item{(ii)} There exists $M=M(n,\lambda,\Lambda)$ such that
 $2 S(x_0,p_0, t ) \subset S(x_0,p_0, M t)\subset M S(x_0,p_0,t)$ for all sections such that $S(x_0,p_0,Mt)\subset \Omega$.
\item{(iii)} There exists $\theta = \theta (n,\lambda,\Lambda) >1$ such that
  if $S(x_0,p_0,2t)\subset \Omega$, then $S(x_0,p_0,t)  \subset S(y,q,\theta t)$ for all  $y \in S(x_0,p_0,t) $ and $q\in \pa\psi(y)$.
\end{proposition}

These properties are classical when $\delta=0$  (see for instance
\cite{Caffarelli91},
\cite{Gutierrez00},
\cite[Chapter 3]{Gutierrez01}), but proving that they hold when $\delta>0$ as long as $t\geq C_0\delta$ is an important contribution of this paper. 
Complete proofs for the case $\delta>0$ are presented in Section \ref{sec:sec}:
Property  (i) will be proved in Proposition \ref{prop:St}.
Since $\omega_n \ell(S)^n \leq |S| \leq\omega_n L(S)^n$, it implies 
\begin{equation}\label{eq:ellt} 
\ell (S) \leq C t^{1/2}, \qquad L(S) \geq C^{-1} t^{1/2}
\end{equation}
for some constant $C = C(n,\lambda,\Lambda)$.
Property  (ii) follows from Proposition \ref{prop:sec1}; note that the second inclusion $S(x_0,p_0, M t)\subset M S(x_0,p_0,t)$ holds for any convex function. 
Finally property  (iii) follows from Proposition \ref{prop:engulf}.

We also make the following important observation:
by iterating Proposition \ref{prop:sec}-(ii), we find $2^k S(x_0,p_0, t ) \subset S(x_0,p_0, M^k t)$ that is 
\begin{equation}\label{eq:secd} 
K S(x_0,p_0, t ) \subset S(x_0,p_0, K^\beta t) \qquad\forall t\geq C_0 \delta \quad \mbox{ if } \quad S(x_0,p_0, K^\beta t) \subset\Omega
\end{equation}
with $\beta = \frac{\ln M}{\ln 2}$.
\medskip

Properties (ii) and (iii) above are required to use Vitali's covering lemma with sections rather than Euclidean balls (see \cite[Section 1.1]{Stein}):
\begin{corollary}[Vitali's covering Lemma] \label{cor:vitali}
Under the assumptions of Proposition \ref{prop:sec}, there exist constants $C_0(n,L(\Omega),L(\pa \psi(\Omega))$ and $C_*(n,\lambda,\Lambda)$  such that given  a collection of sections $S_i=S(x_i,t_i)\subset \Omega$ with $t_i\geq C_0\delta$ 
there exists a countable subcollection $S_{i_j}$ such that the $ S_{i_j}$ are pairwise disjoint and satisfy
$$ \bigcup_{i} S_i \subset \bigcup_j C_* S_{i_j}.$$
\end{corollary}
%
 
\section{$C^{1,\alpha}$ regularity up to scale $\delta$}
It was shown by Forzani and Maldonado \cite{FM} that the engulfing property of sections is actually equivalent to the $C^{1,\alpha}$ regularity of the potential. 
Using this idea, we can prove an approximate $C^{1,\alpha}$ regularity at some scale depending on $\delta$.
The following theorem implies Theorem \ref{thm:alpha}:
\begin{theorem}\label{thm:3}
Let $\Omega\subset \R^n$ be a bounded convex domain 
and  $\psi:\R^n\to \R$ be a convex function satisfying 
\eqref{eq:ineq1}, \eqref{eq:ineq2}, with $\mu$ and $\nu$ satisfying  \eqref{eq:munu1} and \eqref{eq:munu2} with $\delta>0$.
Let $\Omega'\subset\subset \Omega$ be such that
there exists $\rho>C_0\delta $  such that \eqref{eq:rho1} holds.
Given $y\in \Omega'$, 
there exist a constant $\overline C$ depending on 
$\mathrm{dist}(y,\pa\Omega')$, $\rho$, $ \pa\psi(\Omega)$, $\Omega$, $\lambda$, $\Lambda$ and $n$
such that
$$
0\leq  \psi(x) -\psi(y) -q\cdot (x-y) \leq \overline C |x-y|^{1+\frac 1 \theta}  
$$
for all $q\in \pa\psi (y)$ and  $x\in \Omega'$ such that 
\begin{equation}\label{eq:xy}
C\delta^{1/\beta} \leq |x-y|\leq  \min\left(\mathrm{dist}(y,\pa\Omega'),\frac{\rho}{2L(\pa\psi(\Omega))}\right) .
\end{equation}
\end{theorem}
We emphasize that the $C^{1,\alpha}$ regularity of solutions to Monge-Amp\`ere in dimension $n\geq 3$ cannot be derived purely locally but requires appropriate boundary conditions. In Theorem \ref{thm:3}, the role of the boundary condition is played by the assumption that $S(x,p,\rho)\subset \Omega$ for all $x\in \Omega'$ (condition \eqref{eq:rho1}).

To prove Theorem~\ref{thm:3} we start with the following observation.
\begin{lemma}\label{lem:diam}
There exists $C$ and $\beta$ depending on  $\Omega$, $\rho$, $\O $, $n$, $\lambda$, $\Lambda$ such that
for all $x\in \Omega'$ and $p\in \pa\psi(x)$ so that \eqref{eq:rho1} holds we have 
$$ L(S(x,p,t)) \leq C t^\frac{1}{\beta} \qquad \mbox{ for all }\ { t\in [C_0 \delta, \rho]}$$
\end{lemma}
\begin{proof} 
The dilation property of sections \eqref{eq:secd} together with condition \eqref{eq:rho1}  imply
$$ S(x,p,t) \subset \left(\frac t \rho\right)^{\frac 1 \beta} S(x,p,\rho) \subset \left(\frac t \rho\right)^{\frac 1 \beta} \Omega$$ 
and so 
$ L(S(x,p,t)) \leq \frac{L(\Omega) }{\rho^{\frac{1}{\beta}}} t^\frac{1}{\beta}$.
\end{proof}

\bigskip

\begin{proof}[Proof of Theorem \ref{thm:3}]
Given  $x ,y\in \Omega'$ we fix $p\in \pa\psi(x)$ such that \eqref{eq:rho1} holds and choose any $q\in \pa\psi(y)$.
We then take $t$ to be the smallest positive number such that $y\in S(x,p,t)$. 
We then have 
\begin{equation}\label{eq:llt}
 t = \psi(y)-\psi(x)-p\cdot (y-x),
 \end{equation}
and $L(S(x,p,t))\geq |y-x|$.

If $t\leq C_0\delta$, then $|x-y|\leq L(S(x,p,t)) \leq L(S(x,p,C_0\delta))$, and  Lemma \ref{lem:diam} gives $L(S(x,p,C_0\delta)) \leq C(C_0\delta)^\frac{1}{\beta}$. We then have $|x-y|\leq  C(C_0\delta)^\frac{1}{\beta}$ which contradicts \eqref{eq:xy} (for an appropriate choice of the constant $C$.
We must thus have $t\geq C_0\delta$.

Proposition \ref{prop:sec}-(iii) (the engulfing property of sections) then gives
$S(x,t) \subset S(y,q,\theta t) $ provided $2t\leq \rho$ (so that \eqref{eq:rho1} guarantees that $S(x,p,2t)\subset\Omega$). 
 This implies
\begin{equation}\label{eq:lllt}
 \psi(x) \leq \psi(y)+q\cdot(x-y)+\theta t.
 \end{equation}
In view of the definition of $t$, the condition $2t\leq \rho$ is satisfied as long as $y\in S(x,\rho/2)$.
Since $B_{\rho/2L(\pa\psi(\Omega))}(x)\subset S(x,\rho/2)$   (see the proof of Proposition \ref{prop:lL}), this holds as long as $|x-y|\leq \frac{\rho}{2L(\pa\psi(\Omega))}$.

Combining \eqref{eq:llt} and \eqref{eq:lllt}, we deduce
\begin{equation}\label{eq:ff}
\frac{1+\theta}{\theta}[\psi(x)-\psi(y) - q\cdot(x-y)] \leq (p-q)\cdot (x-y)
\end{equation}
which holds for all $x,y \in \Omega'$ such that
$    C\delta^{1/\beta}\leq |x-y|\leq \frac{\rho}{2L(\pa\psi(\Omega))}$.

Let now $v$ be a vector satisfying $|v|=1$ and define
$$ f(s) = \psi(y+sv) -\psi(y) - q \cdot sv.$$
Notice that $\psi$ is convex in $s\in \R$. It is hence Lipschitz regular and differentiable almost everywhere. Hence 
for  $a.e.\;s$ such that 
$$    C\delta^{1/\beta}\leq s \leq s_0= \min\left(\mathrm{dist}(y,\pa\Omega'),\frac{\rho}{2L(\pa\psi(\Omega))}\right) $$ 
inequality \eqref{eq:ff} implies
$$ 
\frac{1+\theta}{\theta} f(s) \leq s f'(s).$$
Define $g(s)= s^{-\frac{1+\theta}{\theta}}f(s)$ which is also Lipschitz on the corresponding interval of $s$ (as it is away from $s=0$) and notice that $g'(s)\geq 0$ for $a.e.\;s$ so that $g$ is monotone increasing. Thus
$$
0\leq 
f(s) \leq   s ^{\frac{1+\theta}{\theta}}g(s_0)  \qquad \forall s\in [C\delta^{1/\beta},s_0],
$$
which gives by taking $x=y+sv$
$$
0\leq  \psi(x) -\psi(y) - q\cdot (x-y) \leq |x-y|^{1+\frac 1 \theta} g(s_0)
$$
for all $x$ such that $C\delta^{1/\beta} \leq |x-y|\leq s_0$.

It remains to see that 
$$g(s_0) =  s_0^{-\frac{1+\theta}{\theta}}f(s_0)\leq  \min\left(\mathrm{dist}(y,\pa\Omega'),\frac{\rho}{2L(\pa\psi(\Omega))}\right) ^{-\frac{1+\theta}{\theta}} L(\pa\psi(\Omega))\, L(\Omega)$$
to get the result.

\end{proof}

\section{The case $\delta =0$: Proof of Theorem \ref{thm:1}}\label{sec:delta0}
In this section we present the strategy for proving the Sobolev regularity of $\psi$ in the simpler case $\delta=0$.
The general argument relies on the ideas introduced in \cite{DF} with only slight modifications. 
Among the noteworthy modifications, we will not rely on  the maximal function (this does not fundamentally change the arguments but makes the proof more self contained). 
We will also avoid the use of renormalization of the sections, something that will be very important when  we extend the proof to the case $\delta>0$ in the next section.

\subsection{Preliminaries}
We consider  a convex function $\psi: \R^n\to\R$ and a section $S(x_0,p_0,t)$ with
 $x_0\in \R^n$ and $p_0\in \pa\psi(x_0)$.
Since we can always replace $\psi$ with $\psi(x)-\psi(x_0) - p\cdot(x-x_0)$, we can assume in this subsection, without loss of generality, that $\psi(x_0)=0$ and $p_0=0$. 
We recall (see Lemma \ref{lem:J}) that up to a translation, we can also assume that there is a positive definite symmetric matrix $A$ such that
\begin{equation}\label{eq:SAS}
\{ x\in \R^n\, ;\, | x| _A \leq 1\}\subset  S(x_0,0,t) \subset \{ x\in \R^n\, ;\, | x| _A \leq n\}  
\end{equation}
and we stress out the fact that we cannot assume that $x_0=0$ since we are already assuming that the ellipsoid in \eqref{eq:SAS} are centered at $0$.

The following  two lemmas, which will play a crucial role in the proof,  are reminiscent of Lemma~3.2 and Lemma 3.3 in \cite{DF}:
\begin{lemma}\label{lem:Asup}
Let $\psi: \R^n\to\R$ be a convex function and $A$ a positive definite symmetric matrix such that $S=S(x_0,0,t)$ satisfies \eqref{eq:SAS}.
There exists a constant $C$ depending only on the dimension $n$ such that
$$
 \fint_ S  \Delta \psi(x) \, dx\leq C \mathrm{Tr} (A) \sup_{\{ | x| _A \leq 2n\}  } \psi
$$
where $\fint f(x)\, dx = \frac{1}{|S|}\int_S f(x)\, dx$.
\end{lemma}

\begin{lemma}\label{lem:Asub}
Let $\psi: \R^n\to\R$ be a convex function and $A$ a positive definite symmetric matrix such that $S=S(x_0,0,t)$ satisfies \eqref{eq:SAS}. There exists 
a constant $c$ depending only on the dimension $n$ and
a  set $V \subset \pa\psi(S)$ such that 
\begin{equation}\label{eq:Vlb}
 |V| \geq c \frac{t^n}{|S|}, \qquad \pa\psi^*(V)\subset S
 \end{equation}
and
\begin{equation}\label{eq:psiA}
\Delta \psi(y_0) \geq \frac{ t}{4n^2}  \Tr(A)  \qquad  \forall y_0\in \Sigma:=\pa\psi^*(V)\subset S.
\end{equation}
Furthermore, for all $q\in V$ and $y_0\in \pa \psi^*(q)$ we have
\begin{equation}\label{eq:psip}
 \psi(y) -\psi(y_0) \geq \frac t {8n^2} | y-y_0|_A^2+q\cdot (y-y_0) \qquad \forall y\in  \{x\in\R^n\, ;\, |x|_A \leq 2n\}.
\end{equation}
\end{lemma}
Note that when  \eqref{eq:psip} holds, the convexity of $\psi$ implies that $\pa\psi^*(q)$ cannot contain any other point besides $y_0$, that is
$$\pa\psi^*(q)  = \{y_0\}.$$
On the other hand it is not necessarily true that 
\eqref{eq:psip} will hold for any another $q'\in \pa\psi(y_0)$. 
This means that we may have  $V \varsubsetneq \pa \psi(\Sigma) $ (even though $\Sigma=\pa\psi^*(V)$).
When $\psi$ is $C^1$ and strictly convex (as is the case in \cite{DF})
this distinction  is not relevant (since $V=\pa \psi(\Sigma)$).
But it constitutes an important difficulty in extending the proof to the case $\delta>0$.
Since we will not be able to assume that $\psi\in C^1$ in that case it will be very important that \eqref{eq:Vlb} provides 
a lower bound on the measure of $V$ rather than the measure of $\pa\psi(\Sigma)$.

An inequality similar to \eqref{eq:psiA} is used in \cite{DF}, and this inequality is enough to prove our main result when $\delta=0$.
Inequality \eqref{eq:psip}, which carries additional nonlocal information, will be needed to extend the result to the case $\delta>0$.

\begin{proof}[Proof of Lemma \ref{lem:Asup}]
We consider
the polynomial $p(x)=2-\frac{x^*Ax}{n^2}$ which,   in view of \eqref{eq:SAS}, satisfies $p(x)\geq 1$ for $x\in S$ and
$\{ p>0\} = \{ x\in \R^n\, ;\, |x|_A \leq \sqrt 2 n \} .$
Let  $\overline \psi := \psi-\sup_{\{ p>0\}} \psi$.
Since $\psi$ is convex, $\Delta \psi=\Delta \overline\psi\geq 0$ and a couple of integrations by parts give
\begin{align*}
 \int_{S} \Delta \psi \, dx  \leq \int_{\{p>0\}} p \Delta \overline \psi \, dx 
&= \int_{\pa \{p>0\}} p\na \overline \psi\cdot n + \int_{\pa \{p>0\}} \overline \psi (-\na p\cdot n) +  \int_{\{p>0\}} \overline \psi \Delta p.
\end{align*}
The first term vanishes since $p=0$ on ${\pa \{p>0\}}$ while the second term is non-positive since $- \na p\cdot n\geq 0$ on ${\pa \{p>0\}}$ and $\overline \psi \leq 0$ in $ \{p>0\} $. We deduce:
\begin{align*}
 \int_{S} \Delta \psi \, dx  \leq \int_{\{p>0\}} \overline \psi \Delta p = \frac{2}{n^2} \mathrm {Tr} (A) \int_{\{p>0\}} (-\overline \psi )
 & \leq \frac{2}{n^2} \mathrm {Tr} (A) \int_{\{p>0\}} [-\inf_{x\in {\{p>0\}} } \overline \psi ]  \\
&\leq  \frac{2}{n^2} \mathrm {Tr} (A)|\{p>0\}| \, \sup_{\{p>0\}} \psi.
\end{align*}
The result follows  since $|\{p>0\}| = (\sqrt2 n)^n |\{  |x|_A \leq 1 \}|\leq C| S|$.
\end{proof}

\begin{proof}[Proof of Lemma \ref{lem:Asub}]
The main idea stems from the proof of Lemma 3.3 in \cite{DF}:
We denote
$ U = \{x\in\R^n\, ;\, |x|_A \leq 2n\}$
and set 
$p(x)= \frac{t}{2} (\frac{1}{4n^2} x^*Ax-1)$.
We have $-\frac t 2 \leq p\leq 0$ in $U$, so that the function  $w(x) := \psi(x)-t-p(x)$  satisfies
$w(x) \geq \psi(x)-t > 0 $ in $U\setminus S$ and $w(x_0) \leq  \psi(x_0) -t + \frac t 2 =- \frac t 2$. 
We then consider $\overline w$ the convex hull of $w$ in $U$, defined by
$$
\overline w(x) = \sup\{ \ell(x)\,;\, \ell\mbox{ affine function s.t. } \ell\leq w \mbox{ in } U, \; \ell\leq 0 \mbox{ on } \pa U\}, 
$$
and  the contact set
$$ \Sigma = \{x\in U \, ;\, \overline w(x)=w(x)\}.$$
Since $w > 0$ in $U\setminus S$ and $\overline w\leq 0$, we have $\Sigma \subset S$. 
Given $y_0\in \Sigma$ and $q_0\in \pa \overline w(y_0)$,  since $\overline w$ is convex, we have
$$w(x)\geq  \overline w(x) \geq \overline w(y_0) + q_0\cdot (x-y_0) = w(y_0) + q_0\cdot (x-y_0) \qquad \forall x\in U.$$
The definition of $w$ and the fact that $p$ is a quadratic polynomial then imply
\begin{align*}
\psi(x) - \psi(y_0) 
& \geq p(x) - p(y_0) +q_0\cdot(x-y_0)  \\
& \geq  \na p(y_0) \cdot (x-y_0) +   \frac 1 2 (x-y_0)^* D^2p(y_0) (x-y_0)+q_0\cdot(x-y_0) \\
& \geq    \frac{t}{8n^2} |x-y_0|^2_A+(q_0+ \na p(y_0) )\cdot(x-y_0) \qquad  \forall x\in U.
\end{align*}
We now denote 
\begin{equation}\label{eq:V}
 V=\{ q_0+\na p(y_0)\, |\, y_0\in \Sigma, \quad  q_0\in \pa\overline w (y_0)\}.
 \end{equation}
The inequality above implies in particular that $q:=q_0+ \na p(y_0)  \in \pa\psi(y_0)$ so that $V\subset \pa \psi (\Sigma)$. As noted earlier, it also implies that $\pa \psi^*(q)=\{y_0\}$.

For all $q\in V$, we can write $q=q_0+\na p(y_0)$ with $y_0\in \Sigma$ (and  $ q_0\in \pa\overline w (y_0)$)
and the computation above shows that \eqref{eq:psip} then holds (and \eqref{eq:psip} implies \eqref{eq:psiA}). We saw above that this also implies that $\pa \psi^*(q)=\{y_0\}$ and so this proves that
\eqref{eq:psip} actually holds for all $q\in V$ and $y_0\in \pa \psi^*(q)$ (as claimed in the lemma).
This also implies that 
$$\pa\psi^*(V) \subset \Sigma.$$
The fact that $\Sigma \subset \pa\psi^*(V) $ is easy to show since for all $y_0\in \Sigma$ (and $q_0\in \pa \overline w(y_0)$) the vector $q=q_0+\na p(y_0)$ satisfies $q\in V$ and $y_0\in \pa\psi^*(q)$. We thus have
$$\Sigma =\pa\psi^*(V) .$$

\medskip

So the only thing left to show to complete the proof of Lemma \ref{lem:Asub} is the lower bound on $|V|$:
Using the convexity of $\overline w $ together with the smoothness of $p$, \eqref{eq:V} implies:
$$ V = \pa(\overline w+p)(\Sigma)$$
and so (since $D^2\overline w+D^2 p \geq D^2\overline w\geq 0$)
$$ |V| = | \pa(\overline w+p)(\Sigma)| = \int_\Sigma \det  (D^2\overline w+D^2 p)\, dx \geq \int_\Sigma \det  (D^2\overline w )\, dx = | \pa\overline w(\Sigma)|  .$$
Furthermore, the definition of the convex hull gives
$\pa \overline w(\Sigma) = \pa \overline w(U)$, hence
$$ |V|  \geq  | \pa\overline w(U)|  .$$
Since $\overline w=0$ on $\pa U$ and $\overline w (x_0)\leq -t/2$, $U$ is a section of $\overline w$ with some center $\bar x_0\in U$ and height $\overline t \geq \frac t 2$. In particular, 
 \eqref{eq:volumesectionbd} implies 
$$ 
 |\pa \overline w (U)|\, |U| \geq \frac{\omega_n^2}{(2n)^n}  \left(\frac t 2\right)^n   $$
 (we can also use Alexandrov's estimate, Proposition \ref{prop:alek})
and so
$$ 
 |V|  \geq  | \pa\overline w(U)|   \geq c t^n \frac{1}{|U|} 
$$
and we conclude our proof by using the fact that $|U|= (2n)^n   |\{  |x|_A \leq 1 \}| \leq (2n)^n |S|$. 
\end{proof}

Lemmas \ref{lem:Asup} and \ref{lem:Asub} hold for any convex function and do not make use of the optimal transportation assumptions. When $\psi$ is an optimal transportation potential satisfying our assumptions, these lemmas yield:
\begin{proposition}\label{prop1}
Let $\psi:\R^n\to \R$ be a convex function satisfying 
 \eqref{eq:O}, \eqref{eq:ineq1}, \eqref{eq:ineq2}, with $\mu$ and $\nu$ satisfying \eqref{eq:munu1} and \eqref{eq:munu2} with $\delta=0$.
There exists some constants $c$ and $C_1$ depending on $n$, $\lambda$ and  $\Lambda $ such that the following holds:
Given $x_0\in \Omega$, $p_0\in \pa\psi(x_0)$, $t$ such that $S(x_0,p_0,C_1 t)\subset\Omega $ and setting
$$\alpha:=\fint_{S(x_0,p_0,t)}\Delta \psi (x)\, dx ,$$
there exists $\Sigma\subset S(x_0,p_0,t)$ such that
$$\Delta \psi(x) \geq c\alpha \quad \forall x\in \Sigma, \quad \mbox{ and }\quad  |\Sigma|\geq c|S(x_0,p_0,t)|.$$
\end{proposition}

\begin{proof}
Up to subtracting a linear function to $\psi$, we can assume that
$$p_0=0, \quad\psi(x_0)=0$$
so that $0\leq \psi(x)\leq t$ in $S(x_0,0,t)$ and $\psi(x)=t$ on $\pa S(x_0,0,t)$.
Up to a translation, Lemma \ref{lem:J} implies the existence of a matrix $A$ such that  \eqref{eq:SAS} holds.
Lemma \ref{lem:Asup} then implies
$$
\fint_{S(x_0,0,t)}
\Delta \psi(x)\, dx \leq  C \mathrm{Tr} (A) \sup_{\{|x|_A\leq 2n\}} \psi
$$
Using \eqref{eq:secd}, we find $\{|x|_A\leq 2n\}=2n \{|x|_A\leq 1\}\subset 2 n S(x_0,0,t) \subset S(x_0,0,C_1 t )$ with $C_1= (2n)^\beta$ 
provided $ S(x_0,0,C_1 t)\subset \Omega$.
It follows that $ \sup_{\{|x|_A\leq 2n\}}\psi \leq C_1 t$ and so
\begin{equation}\label{eq:alphat} 
\alpha = \fint_{S(x_0,0,t)} \Delta \psi(x) \, dx  \leq C \mathrm {Tr}(A)  t   .
\end{equation}
for some $C$ depending only on $n$, $\lambda$ and $\Lambda$.

We now use Lemma \ref{lem:Asub}  to deduce the existence of a set $\Sigma =\pa\psi^*(V)\subset S(x_0,0,t)$ such that
$$\Delta \psi(x) \geq \frac{t}{4n^2} \mathrm {Tr}(A) \quad\mbox{ a.e.  in } \Sigma , \qquad |V| \geq c \frac{t^n}{|S|}.$$
Together with \eqref{eq:alphat}, this implies 
$$ \Delta \psi(x) \geq c \alpha\quad \mbox{ a.e.  in  } \Sigma,$$
for some $c$ depending only on $n$, $\lambda$ and $\Lambda$.
Using the lower bound on $|V|$ together with \eqref{eq:ineq2} (note that $V \subset \pa\psi(S) \subset \O $ since $\psi$ satisfies \eqref{eq:O}) and Assumption \ref{ass:1}, we get
$$
c \lambda \frac{t^n}{|S|} \leq \lambda |V| \leq \nu(V) \leq \mu(\pa\psi^*(V)) = \mu(\Sigma) \leq \Lambda |\Sigma|.
$$
The volume inequality of Proposition \ref{prop:sec} (i)  give $\frac{t^n}{|S|}\geq c|S|$, hence
$$ |\Sigma| \geq c |S|$$
for some constant $c$ depending on $n$, $\lambda$ and $\Lambda$.
\end{proof}

\subsection{A reversed Chebychev's inequality}\label{sec:cheb}
The heart of the proof of Theorem \ref{thm:1} is the following reversed Chebychev's inequality
(which is equivalent to (3.15) in \cite{DF}):
\begin{theorem}\label{thm:cheby}
Let $\psi:\R^n\to \R$ be a convex function satisfying \eqref{eq:O},
\eqref{eq:ineq1}, \eqref{eq:ineq2}, with $\mu$ and $\nu$ satisfying  \eqref{eq:munu1} and \eqref{eq:munu2} with $\delta=0$. Assume moreover that $\Delta\psi\in L^1_{loc}$. 
Let $\Omega' \subset\subset \Omega$ be such that there exists $\rho>0$  such that \eqref{eq:rho1} holds.
Then there exists $\alpha_0$, $c$ and $C$  depending on 
$\rho$, $n$, $\lambda$, $\Lambda$, $\Omega$ and $\pa\psi(\Omega)$
such that
\begin{equation}\label{eq:cheby}
\int_{\{x\in \Omega'\, |\, \Delta\psi(x)\geq \alpha \}} \Delta\psi(x)\, dx \leq C \alpha |  \{x\in \Omega\,|\,   \Delta\psi(x) \geq c\alpha\}| \qquad\mbox{ for all $\alpha>\alpha_0$}.
\end{equation}
\end{theorem}
Let us first show that this theorem implies Theorem \ref{thm:1} (the proof is similar to the proof in \cite{DF}. We recall it here for the sake of completeness). 
\begin{proof}[Proof of Theorem \ref{thm:1}]
For $\bar c\geq \max\{2,\alpha_0\}$ (so that $\log(2+\gamma) \leq 2\log \gamma$ for $\gamma\geq \bar c $), we can write
\begin{align*}
\int_{\Omega'} \Delta\psi(x)\log(2+\Delta\psi(x))\, dx
& \leq \bar c \log(2+\bar c) |\{x\in \Omega'\,|\, \Delta\psi(x)\leq \bar c\}|\\
& \qquad\qquad 
+ 2 \int_{\{x\in \Omega'\,|\, \Delta\psi(x)\geq \bar c\}} \Delta\psi(x)\log(\Delta\psi(x))\, dx\\
& \leq \bar c \log(2+\bar c) |\Omega| 
+ 2 \int_{\{x\in \Omega'\,|\, \Delta\psi(x)\geq \bar c\}} \Delta\psi(x)\int_1^{\Delta\psi(x)} \frac 1 \gamma\, d\gamma \, dx\\
& \leq \bar c \log(2+\bar c) |\Omega| + 2 \log \bar c \int_{  \Omega' } \Delta\psi(x)\, dx\\
&\qquad + 2 \int_{\{x\in \Omega'\,|\, \Delta\psi(x)\geq \bar c\}} \Delta\psi(x)\int_{\bar c}^{\Delta\psi(x)} \frac 1 \gamma\, d\gamma \, dx
\end{align*}
Now, by Fubini's theorem, the last term can be written as
$$
\int_{\{x\in \Omega'\,|\, \Delta\psi(x)\geq \bar c\}} \Delta\psi(x)\int_{\bar c}^{\Delta\psi(x)} \frac 1 \gamma\, d\gamma \, dx
=
\int_{\bar c}^\infty \frac{1}{\gamma} \int_{\{x\in \Omega'\,|\, \Delta\psi(x)\geq \gamma \}\}}\Delta\psi(x)\, dx\, d\gamma$$
and we can use Theorem \ref{thm:cheby}: Since $\gamma>\bar c >\alpha_0$, we get
\begin{align*}
\int_{\bar c}^\infty \frac{1}{\gamma} \int_{\{x\in \Omega'\,|\, \Delta\psi(x)\geq \gamma \}\}}\Delta\psi(x)\, dx\, d\gamma
& \leq C \int_{\bar c}^\infty  |\{x\in \Omega\,|\, \Delta\psi(x)\geq c  \gamma\}|\, d\gamma\\
& \leq \frac{C}{c} \int_{c \bar c}^\infty  |\{x\in \Omega\,|\, \Delta\psi(x)\geq  \gamma\}|\, d\gamma\\
& \leq \frac{C}{c} \int_\Omega  \Delta\psi(x)\, dx\\
\end{align*}
Putting things together, we proved that for some constant $C$ we have
$$
\int_{\Omega'} \Delta\psi(x)\log(2+\Delta\psi(x))\, dx
\leq C + C \int_{  \Omega} \Delta\psi(x)\, dx$$
and we conclude using the fact that 
$$ \int_{  \Omega} \Delta\psi(x)\, dx \leq  \int_{  \Omega} \Delta \psi \, dx = \int_{\pa\Omega} \na \psi \cdot n \, d \mathcal H^{n-1}(x)\leq \mathcal H^{n-1}(\pa\Omega ) L(\pa\psi(\Omega)).$$
\end{proof}

The proof of Theorem~\ref{thm:cheby} is where we explicitly use the $L^1$ regularity of $\Delta\psi$. More precisely we use the fact that the map $t\mapsto \fint_{S(x,t)} \Delta\psi(x)\,dx$ is continuous and satisfies
\[
\lim_{t\to0}  \fint_{S(x,t)} \Delta\psi(x)\,dx=\Delta\psi(x)\ a.e.
\]
As remarked in the introduction, it should be possible to remove this assumption by correctly decomposing $\Delta\psi$ into a regular and a singular part in the argument. Since this more general case is thoroughly covered by our theorem in the case $\delta>0$, we stay at the more elementary level where $\Delta\psi$ is absolutely continuous w.r.t. the Lebesgue measure.

\begin{proof}[Proof of Theorem \ref{thm:cheby}]
We fix some $\eps>0$ (to be chosen later) and take $\alpha_0$ such that
$$ \fint_{S(x,t)} \Delta\psi(x)\,dx \leq \alpha_0 \qquad \forall x\in \Omega', \; \forall t\in[\eps\rho,\rho].$$
The existence of such an $\alpha_0$ follows from the fact that for $ t\in[\eps\rho,\rho]$ we have
$$\fint_{S(x,t)} \Delta\psi(x)\,dx \leq \frac{1} {|S(x,t)|} \int_{\Omega} \Delta\psi(x)\,dx\leq 
\frac{1}{|S(x,\eps\rho)|} \int_{\Omega} \Delta \psi \,dx 
\leq \frac{\mathcal H^{n-1}(\pa\Omega)L(\pa\psi(\Omega)) }{|S(x,\eps\rho)|}  ,$$ 
where $|S(x,\eps\rho)| \geq c (\eps\rho)^{n/2}$ (see Proposition \ref{prop:sec} (i)).

We now take $\alpha>\alpha_0$ and denote
\[
E_\alpha = \{x\in \Omega'\, |\, \Delta\psi(x)>\alpha\}.
\]
For a.e. $x\in E_\alpha$ we have $ \fint_{S(x,t)} \Delta \psi(x) \,dx< \alpha$ for all $t\in[\eps\rho,\rho]$ (by the choice of $\alpha$) and 
$$\lim_{t\to0}  \fint_{S(x,t)} \Delta\psi(x)\,dx=\Delta\psi(x)> \alpha$$
by the Lebesgue differentiation theorem.

We can thus define
$$ t(x) = \sup\left\{ t\leq \rho\, ;\,  \fint_{S(x,t)} \Delta\psi(x)\,dx = \alpha\right\}$$
which satisfies in particular $t(x)\leq \eps \rho$ and 
\begin{equation}\label{eq:Stt}
  \fint_{S(x,t)} \Delta\psi(x)\,dx \leq \alpha \quad\mbox{ for all } t(x) \leq t \leq \rho, \qquad  \fint_{S(x,t(x))} \Delta\psi(x)\,dx = \alpha.
\end{equation}
Since $E_\alpha \subset \bigcup_{x\in E_\alpha} S(x,t(x))$, 
we can use Vitali's covering lemma (see Corollary \ref{cor:vitali}) to extract a countable family $\{x_k\}_{k\in \mathbb N}$ such that the sections $S_k = S(x_k,t(x_k))$ are pairwise disjoint and  $E_\alpha \subset \bigcup_{k}  C_{*} S_k$ with a  $C_*$ depending only on $n$, $\lambda$ and $\Lambda$.
 
Using \eqref{eq:secd}, we get
$ C_{*} S_k \subset S(x_k,C_{**}   t_k) $ with $C_{**} = C_*^\beta$, 
provided $S(x_k,C_{**}   t_k) \subset \Omega$, which holds if $C_{**}   t _k\leq \rho$ (since $x_k\in \Omega'$).
Since we know that $t _k\leq \eps\rho$, this last condition can be enforced by 
choosing $\eps: = \frac1 { C_{**} } $.

Note that we then also have $t_k\leq C_{**}  t_k\leq \rho$.
The first inequality in \eqref{eq:Stt} thus implies
$$  \int_{C_{*} S_k} \Delta\psi(x)\,dx  \leq \int_{S(x_k,C_{**}  t_k)} \Delta\psi(x)\,dx \leq \alpha | S(x_k,C_{**}   t_k)|.$$
Proposition \ref{prop:sec} (i) gives $| S(x_k,C_{**}   t_k)|\leq C  t_k^\frac{n}{2} \leq C |S(x_k,t_k)|$ hence
$$
 \int_{C_{*} S_k} \Delta\psi(x)\,dx\leq C\alpha |S_k|.
 $$
We are ready to conclude: First, we write (since $E_\alpha \subset \bigcup_{k}  C_{*} S_k$)
\begin{align*}
\int_{E_\alpha} \Delta\psi(x)\, dx
& \leq \sum_k \int_{  C_{*} S_k} \Delta\psi(x)\, dx \leq C\alpha  \sum_k | S_k| .
\end{align*}
Then, we apply Proposition \ref{prop1}: We note that $S(x_k,C_1 t_k) \subset S(x_k,C_1 \eps \rho) \subset \Omega$ (up to taking a smaller $\eps$ so that $\eps \leq 1/C_1$) and $ \fint_{S_k} \Delta\psi(x)\, dx = \alpha$ (see \eqref{eq:Stt}). Proposition \ref{prop1} thus provides the existence of a subset $\Sigma_k\subset S_k$ with $|\Sigma_k|\geq c|S_k|$ and  $\Delta \psi\geq c\alpha$ a.e. in $\Sigma_k$. 
We can then write:
\begin{align*}
\int_{E_\alpha} \Delta\psi(x)\, dx
& \leq C\alpha  \sum_k | \Sigma_k|  \\
& \leq C\alpha  |\cup_k  \Sigma_k|  \\
& \leq C\alpha|  \{ x\in \Omega \,|\, \Delta \psi(x)\geq c\alpha\}|
\end{align*}
where the second  inequality follows from the fact that the $S_k$, and therefore the $\Sigma_k$, are disjoint.
\end{proof}
\section{The case $\delta >0$: Proof of Theorem \ref{thm:2}}\label{sec:delta}
We now turn to the proof of our main result when $\delta>0$. 
We recall that we fix a kernel  $K:\R^n \to [0,\infty]$ radially symmetric, supported in the ball $B_1(0)$ and such that $\int_{B_1} |x|^2 K(x) \, dx>0$ and we define for any $r>0$
$$ 
\Delta_r\psi (x):= \frac{m_0}{r^2}   \int K_r(y) [\psi(x+y) -\psi(x)]\, dy, \qquad K_r(x) = \frac 1 {r^n} K\left(\frac x r\right)
$$
with $m_0=\frac{2}{ \int_{\R^n} y_1^2 K(y) \, dy}$.

We note that
\begin{enumerate}
\item For any $p\in \pa\psi(x)$, the radial symmetry of $K$ allows us to write
\begin{equation}\label{eq:Dp1}
\Delta_r\psi (x)= \frac{m_0}{r^2}   \int K_r(y) [\psi(x+y) -\psi(x) - p\cdot y]\, dy  
\end{equation}
This formula makes the definition of $\Delta_r$ useful when working with sections of the function $\psi$. It also implies that 
$\Delta_r\psi (x)\geq 0$ for all $x\in \R^n$ for any convex function. 
\item If $p(x)$ is an homogeneous polynomial of degree $2$, that is $p(x) = x^* A x$ for some symmetric matrix $A$, then we have (using the radial symmetry of $K$ again):
\begin{equation}\label{eq:Dp2}
\Delta_r p(x) =   \frac{m_0}{r^2}  \sum_{i,j} \int K_r(y) A_{ij} y_i y_j \, dy
=  \frac{m_0}{r^2}  \sum_{i}  A_{ii}  \int K_r(y) y_i^2  \, dy = 2 \mathrm{Tr} A = \Delta p(x).
\end{equation}
The choice of the normalization constant $m_0$ is naturally critical here.
\item Together with Taylor's formula, \eqref{eq:Dp2} also implies that if $\psi$ is  $C^2$ then
$$ 
\Delta_r\psi (x) = \frac{m_0}{r^2}  \sum_{i,j} \int K_r(y) \frac 1 2 \pa_{ij}\psi(x) y_i y_j \, dy + o(1)
= \Delta \psi(x) + o(1).
$$
\end{enumerate}
\subsection{Preliminary results}
Our first task it to generalize Lemmas \ref{lem:Asup} and \ref{lem:Asub} 
using $\Delta_r$ instead of $\Delta$.
The generalization of Lemma \ref{lem:Asup} is relatively straightforward (see Lemma \ref{lem:supd} below), 
and the main issue with  Lemma \ref{lem:Asub} will be to fatten the set $\Sigma$ so that 
it is big enough to allow us to use the estimates \eqref{eq:munu1}-\eqref{eq:munu2} when $\delta>0$ (see Propositions \ref{prop:Vh} and   \ref{prop2}).

\medskip

Given a convex function $\psi:\R^n \to \R$, we fix
 $x_0\in \R^n$ and $p_0\in \pa\psi(x_0)$.
Since we can always replace $\psi$ with $\psi(x)-\psi(x_0) - p\cdot(x-x_0)$, we can assume, without loss of generality, that $\psi(x_0)=0$ and $p_0=0$. 
We recall (see Lemma \ref{lem:J}) that up to a translation, we can also assume that there is a positive definite symmetric matrix $A$ such that
\begin{equation}\label{eq:SAS1}
\{ x\in \R^n\, ;\, | x| _A \leq 1\}\subset  S(x_0,0,t) \subset \{ x\in \R^n\, ;\, | x| _A \leq n\}  .
\end{equation}

We start with the following generalization of Lemma \ref{lem:Asup}:
\begin{lemma}\label{lem:supd}
Let $\psi:\R^n \to \R$ be a convex function and $A$ be a positive definite symmetric matrix $A$ such that
the section $S= S(x_0,0,t)$ satisfies \eqref{eq:SAS1}.
There exists a constant $C$ depending only on $n$ such that \begin{equation}\label{eq:SDr}
  \fint_ S  \Delta_r  \psi (x) \, dx\leq C \mathrm{Tr} (A) \sup_{\{|x|_A\leq 2 n \}+B_r} \psi.
\end{equation}
In particular, if $r \leq \ell(S)$, then
$$
  \fint_ S  \Delta_r  \psi (x) \, dx\leq C \mathrm{Tr} (A) \sup_{ \{|x|_A\leq 3n \}} \psi.
$$
\end{lemma}
\begin{proof}
We introduce the polynomial $p(x)=2-\frac{x^*Ax}{n^2}$ which satisfies (see \eqref{eq:SAS1})
$$\{ p>0\} = \{ x\in \R^n\, ;\, x^*Ax < 2n^2\}, $$
and  $p\geq 1$ in $S$.

We denote $U = \{ p>0\} +B_r $ and 
$\overline \psi := \psi -\sup_{U} \psi$. Since $\Delta_r \psi  = \Delta_r  \overline \psi$, 
a  simple change of variable gives
\begin{align*}
 \int_{S} \Delta_r \psi  \, dx  \leq \int_{\{p>0\}} p(x) \Delta_r  \overline \psi (x)\, dx 
& =\int_{\{p>0\}} \overline \psi (x)  \Delta_r p(x) \, dx  \\
& \quad + \frac{m_0}{r^2} \int_{  \{p>0\}^c } \int_{ \{p>0\} } K_r(x-y) p(y) \overline \psi(x) \, dy\, dx\\
& \quad - \frac{m_0}{r^2}  \int_{ \{p>0\} } \int_{ \{p>0\} ^c} K_r(x-y) p(y) \overline \psi(x) \, dy\, dx.
\end{align*}
Since $\mathrm{Supp} \, K_r \subset B_r$, we can write the last two terms as
$$
 \frac{m_0}{r^2} \int_{U\setminus  \{p>0\} } \int_{ \{p>0\} } K_r(x-y) p(y) \overline \psi(x) \, dy\, dx - \frac{m_0}{r^2}  \int_{ \{p>0\} } \int_{U\setminus  \{p>0\} } K_r(x-y) p(y) \overline \psi(x) \, dy\, dx
$$
 which is negative since $ \overline \psi(x)\leq 0$ in $U$, $p(y)>0$ in $ \{p>0\}$ and $p(y)<0$ in $U\setminus  \{p>0\} $.

 Using \eqref{eq:Dp2}, we deduce
\begin{align*}
 \int_{S} \Delta_r \psi  \, dx  & \leq
 \int_{\{p>0\}} \overline \psi (x)  \Delta_r p(x) \, dx\\
&= - \frac{2}{n^2} \mathrm {Tr} (A) \int_{\{p>0\}} \overline \psi (x)\, dx \\
&\leq  - \frac{2}{n^2}  \mathrm {Tr} (A)|\{p>0\}| \inf _{{\{p>0\}} } \overline \psi   .
\end{align*}
Since $ \inf _{{\{p>0\}} } \overline \psi   =  \inf _{{\{p>0\}} } \psi-\sup_{U} \psi \geq -\sup_{U}\psi$ (recall that $\psi\geq 0$ in $\R^n$) and $|\{ p>0\}| \leq C | S|$, we deduce
$$ \int_{S} \Delta_r \psi  \, dx  \leq C \mathrm {Tr} (A) |S| \sup_{U}\psi$$
for some constant $C$ depending only on $n$. The result follows.
\end{proof}
On the other hand, Lemma \ref{lem:Asub} would easily imply
\begin{lemma}\label{lem:Asubd}
Let $\psi:\R^n \to \R$ be a convex function and $A$ be a positive definite symmetric matrix $A$ such that
the section $S= S(x_0,0,t)$ satisfies \eqref{eq:SAS1}.
There exists 
a constant $c$ depending only on the dimension, a subset $V\subset\pa\psi(S) $ such that $ \pa\psi^*(V)\subset S$ and  
\begin{equation}\label{eq:papsij} |V| \geq c \frac{t^n}{|S|}
\end{equation}
and   for all $r \leq \ell(S)$ we have
\begin{equation}\label{eq:Drpsil}
\Delta_r \psi(y_0) \geq \frac t {4n^2}  \,\mathrm{Tr}(A)   \qquad  \forall y_0\in \Sigma:=  \pa\psi^*(V).
\end{equation}
\end{lemma}
\begin{proof}
Lemma \ref{lem:Asub} implies the existence of the set $V\subset \pa\psi(S) $ satisfying $ \pa\psi^*(V)\subset S$ and \eqref{eq:papsij}
and such that for all $y_0\in\Sigma$, we have 
\begin{equation}\label{eq:psilA}
 \psi(y) -\psi(y_0) \geq \frac t {8n^2} |y-y_0|_A^2+q\cdot (y-y_0) \qquad \forall  y \in  \{x\in\R^n\, ;\, |x|_A \leq 2n\}
\end{equation}
for some  $q\in \pa \psi(y_0)$.

If $r\leq \ell(S)$ then for any $y_0\in \Sigma \subset S$ we have $B_r(y_0)\subset 2S \subset \{|x|_A \leq 2n\}$. Since $K_r$ is supported in $B_r$, we can multiply \eqref{eq:psilA} by $\frac{m_0}{r^2}   K_r(y_0-y)$ and integrate with respect  to $y$ to get (using the radial symmetry of $K_r$ and  \eqref{eq:Dp2}):
\begin{align*}
\Delta_r \psi(y_0)
&\geq \frac{m_0}{r^2}  \int 
K_r(y_0-y) \left[ \frac t {8n^2} |y-y_0|_A^2+q\cdot (y-y_0) \right]\, dy 
 \nonumber \\
& \geq  \frac t {8n^2} \frac{m_0}{r^2}  \int 
K_r(y) y^*Ay \, dy \nonumber \\
 & \geq \frac t {4n^2}  \,\mathrm{Tr}(A) .
\end{align*}
\end{proof}
This result  is however not enough for our purpose
because \eqref{eq:papsij} does not yield any control on $|\Sigma|$ (we  can have $|\Sigma|=0$, even if  $|V |>0$).
We thus need to "fatten" the set $\Sigma$ by showing that \eqref{eq:Drpsil} holds on a slightly larger set.
This will be done by considering the set
$$ \Sigma_h = \bigcup_{q_0\in V, y_0\in \pa\psi^*(q_0)} S(y_0,q_0,h)$$
for some appropriate $h$. 
\medskip

We emphasize that the choice of $\Sigma_h$ is delicate. There could {\em a priori} be many ways of enlarging $\Sigma$, for example by simply taking $\Sigma+B_h$. However we need several precise properties on $\Sigma_h$ (see Prop.~\ref{prop:Vh} later) that we are not able to prove on $\Sigma+B_h$. Roughly speaking, we need to preserve some equivalent of the duality that we have between $V$ and $\Sigma$. The definition above of $\Sigma_h$ will naturally lead to a set $V_h$ that properly enlarges $V$ and maintains the inclusion $\pa\psi^*(V_h)\subset \Sigma_h$.

\medskip

First, we need to prove  the following lemma which will allow us to extend the inequality  \eqref{eq:Drpsil} to the larger set $\Sigma_h$: 
\begin{lemma}\label{lem:DrS}
For all $q_0\in V$ and $y_0\in\pa\psi^*(q_0)$, if $x$ is such that 
$$B_r(x) \subset \{x\, ;\, |x|_A\leq 2n\} ,$$
then
$$
\Delta_r \psi(x)\geq  \frac{t}{4n^2} \mathrm{Tr}(A)  -   \frac{m_0}{r^2} (q-q_0)\cdot(x-y_0) \qquad \forall q\in\pa\psi(x).
$$
\end{lemma}
\begin{proof}
We fix $q_0\in V$ and $y_0\in\pa\psi^*(q_0)$ (we recall that $\pa\psi^*(q_0) = \{ y_0\}$ for $q_0\in V$).
If $x$ is such that $B_r(x) \subset \{x\, ;\, |x|_A\leq 2n\} $, then \eqref{eq:psip} applied to $x+y$ implies that
$$
 \psi(x+y) -\psi(y_0) \geq \frac t {8n^2} |x+y-y_0|_A^2+q_0\cdot (x+y-y_0) \qquad \forall  y \in  B_r(0).
$$
Given $q\in \pa \psi(x)$, we have $\psi(y_0) \geq \psi(x) + q\cdot(y_0-x)$ and so
\begin{align*}
 \psi(x+y) -\psi(x) 
 &\geq \frac t {8n^2} |x+y-y_0|_A^2+q_0\cdot (x+y-y_0) + \psi(y_0)-\psi(x)\\
 &\geq \frac t {8n^2} |x+y-y_0|_A^2+q_0\cdot (x+y-y_0) + q\cdot(y_0-x)\\
 &\geq \frac t {8n^2} |x+y-y_0|_A^2+q_0\cdot y  - (q-q_0)\cdot(x-y_0)
\qquad \forall  y \in  B_r(0).
\end{align*}
We deduce
\begin{align*}
\Delta_r \psi(x) & = \frac {m_0}{r^2} \int K_r(y) [\psi(x+y)-\psi(x)]\, dy \\
& \geq\frac{t}{8n^2}  \frac {m_0}{r^2} \int K_r(y)\left[ |x-y_0|_A^2+ |y|_A^2  \right]\, dy  -  \frac{m_0}{r^2} (q-q_0)\cdot(x-y_0)\\
& \geq  \frac{t}{4n^2} \mathrm{Tr}(A) + 
\frac {m_0}{r^2} \frac{t}{8n^2}|x-y_0|_A^2 -   \frac{m_0}{r^2} (q-q_0)\cdot(x-y_0),
\end{align*}
where we used the symmetry of $K_r$ and the normalization $\int K_r(y)\, dy=1$.
\end{proof}

\noindent{\bf The sets $V_h$ and $\Sigma_h$:}
For $h>0$ and $q_0\in V$, we introduce the set
\begin{align*} 
W_h(q_0) & := q_0 + \frac h 2 (S(y_0,q_0,h)-y_0)^\circ \\
& = \{ q\in \R^n\, ;\, (q-q_0)\cdot (x-y_0)\leq \frac h 2 , \; \forall x\in S(y_0,q_0,h)\} 
\end{align*}
where $\{y_0\}  =\pa \psi^*(q_0)$. 
We note that the polar body $(S(y_0,q_0,h)-y_0)^\circ$ of a convex set (see definition \eqref{def:polar}) is convex and contains $0$.
So $W_h(q_0)$ is expected to be a small convex set around $q_0$.
We then introduce the sets
$$ V_h = \bigcup_{q_0\in V} W_h(q_0)$$
and 
$$ \Sigma_h = \bigcup_{y_0\in \Sigma} S(y_0,q_0,h).$$

The properties of these sets are summarized in the proposition below.
\begin{proposition}\label{prop:Vh}
The following properties hold for some $k$ depending only on $n$
\item[(i)] $\pa \psi^*(W_h(q_0)) \subset S(y_0,q_0,h)$ and so $\pa\psi^*(V_h) \subset \Sigma_h$
\item[(ii)] $\ell (S(y_0,q_0,h)) \geq \frac{h}{L(\pa\psi(\Omega))}$ and  $ \ell (W_h(q_0)) \geq \frac h {4n L(\Omega)}$
\item[(iii)] If $t\geq C_0 \delta$, $h\leq Mt$ and $S(y_0,q_0,2M^kt) \subset \Omega$ for all $y_0\in \Sigma$, then 
$$ \Sigma_h \subset \theta M^k S(x_0,0,t)$$
\item[(iv)] For all $x \in S(y_0,q_0,h)$ and  $q\in \pa \psi(x) $ we have
$$ (x-y_0)\cdot (q-q_0)  \leq M^k h $$
where the constants $C_0$, $M$ and $\theta$ are the constant appearing in Proposition \ref{prop:sec}.
\end{proposition}
Before proving this  proposition, we state the following lemma which follows from Proposition \ref{prop:secpolar}, but for which we also give a short proof below. 
\begin{lemma}\label{lem:jh}
Let $\psi$ be a convex function and $S=S(x_0,p_0,t)$ be a section of $\psi$.
If $x \in S$ is such that $2x-x_0\in S$ (which means that $x\in \frac 1 2 S$ where the contraction is done with respect to $x_0$) then
$$ (x-x_0)\cdot (p-p_0) \leq t \qquad \forall p\in \pa\psi(x)$$
and 
$$ x_0 \in S(x,p,t).$$
\end{lemma}
\begin{proof}
Without loss of generality, we can assume that $x_0=0$, $p_0=0$ and $\psi(x_0)=0$. We  are then assuming $x\in S$ and $2x\in S$, which gives $\psi(2x)\leq t$. The convexity of $\psi$ implies
$$t\geq  \psi(2x) \geq \psi(x) + p\cdot x \geq p\cdot x\qquad \forall p\in \pa\psi(x)$$
which is the first inequality in the lemma.
We can then write
$$\psi(0)\leq \psi(x) \leq \psi(x) -p\cdot x +t$$
which implies $0\in S(x,p,t)$ and concludes the proof.
\end{proof}

\begin{proof}[Proof of Proposition \ref{prop:Vh}]
The first property is an immediate consequence of the definition of $W_h(q_0)$ since 
Proposition \ref{prop:secpolar} (see \eqref{eq:SoS1}) implies
$$ \pa \psi^*(W_h(q_0)) \subset S(y_0,q_0,h).$$
Next, we note that Proposition \ref{prop:lL} gives $\ell (S(y_0,q_0,h)) \geq \frac{h}{L(\pa\psi(\Omega))}$ which is the first inequality in (ii) and  Lemma \ref{lem:radK} yields
$$ \ell (W_h(q_0)) = \frac h 2 \ell((S(y_0,q_0,h)-y_0)^\circ) \geq \frac h {4n L(S(y_0,q_0,h))} \geq \frac h {4n L(\Omega)}$$ 
which is the second inequality in (ii).

\medskip

Lemma \ref{lem:jh} together with the scaling and engulfing property of sections implies (iii). Recall that the section $S(x_0,0,t)$ is included in the ellipsoid $\{(x-x_0^*)^*\,A\,(x-x_0^*)\leq n^2\}$ for $x_0^*$ the center of mass. Since the ellipsoid $\{(x-x_0^*)^*\,A\,(x-x_0^*)\leq 1\}$ is included in the section, this implies that $2x_0^*-x_0\in n\,S(x_0,0,t)$.
Since $t\geq C_0\,\delta$ the scaling property of sections, Proposition \ref{prop:sec}-(ii), gives that  $S(x_0,0,t) \subset \frac{1}{3n} S(x_0,0,M^k t)$ for some $k$ depending only on $n$. Hence if $y_0\in S(x_0,0,t)$,
\[
2y_0-x_0=x_0^*+2(y_0-x_0^*)-(x_0-x_0^*)=x_0^*+3\,\left(\frac{2}{3}\,y_0 +\frac{1}{3}\,(2\,x_0^*-x_0)-x_0^*\right)\in 3\,n\,S(x_0,0,t).
\]
Hence we satisfy the assumption of Lemma \ref{lem:jh} for $S=S(x_0,0,M^k\,t)$ implying that $x_0 \in S(y_0,q_0,M^k\,t)$.
The engulfing property of sections, namely Proposition \ref{prop:sec}-(iii),  implies
$$ y_0\in S(y_0,q_0,M^k t) \subset S(x_0,0,\theta M^k t). $$
Since this holds for all $y_0\in \Sigma$ as $\Sigma\subset S(x_0,0,t)$, it follows that as long as $h\leq Mt$ we have 
$$\Sigma_h\subset S(x_0,0,\theta M^k t )\subset \theta\, M^k\, S(x_0,0,t ).$$

\medskip

Finally, if $x\in S(y_0,q_0,h)$ then the scaling property of sections again gives $x \in \frac{1}{3n}\, S(y_0,q_0,M^kh)$ and so 
Lemma \ref{lem:jh} implies (iv).
\end{proof}

\noindent{\bf The choice of $h$:} 
For any $x\in \Sigma_h$, we have $x\in S(y_0,q_0,t)$ for some $q_0\in V$, $y_0\in \Sigma$. So Lemma
\ref{lem:DrS} together with Proposition \ref{prop:Vh} implies that for some constant $C$
$$
\Delta_r \psi(x)\geq  \frac{t}{4n^2} \mathrm{Tr}(A)  -  C m_0 \frac{h}{r^2}\qquad \forall x\in \Sigma_h.
$$
We thus define
$$ h_0 :=  \frac{t }{8n^2 C m_0} r^2 \mathrm{Tr}(A), \qquad \widetilde \Sigma : = \Sigma_{h_0} $$
so that we have:
$$ \Delta_r \psi(x) \geq  \frac t {8n^2}  \,\mathrm{Tr}(A) \qquad \forall x\in \widetilde \Sigma$$
which is the most important property of $\widetilde \Sigma$.
In addition, unlike $\Sigma$, the set $\widetilde \Sigma$ is "large" enough for us to use Assumption \ref{ass:1} to control its size.
Indeed, 
we recall that
\begin{equation}\label{eq:SV}
 \widetilde  \Sigma = \bigcup_{y_0\in \Sigma} S(y_0,q_0,h_0),\quad \tilde V = \bigcup_{q_0\in V} W_{h_0}(q_0),
 \end{equation}
and since $\mathrm{Tr}(A) \geq \frac 1 {\ell(S)^2}$, we have
$ h_0 \geq   c  \frac{t}{\ell(S)^2} r^2  .$
In particular, Proposition \ref{prop:Vh} (ii) implies that if 
$$ r \geq \delta^{1/2} \quad\mbox{ and } \quad \frac{t}{\ell(S)^2} \geq C$$
then
\begin{equation}\label{eq:SW}
 \ell (S(y_0,q_0,h_0)) \geq \delta \quad \mbox{  and  } \quad  \ell (W_{h_0}(q_0)) \geq \delta.
 \end{equation}

We can now prove following proposition, which generalizes Proposition \ref{prop1} to the case $\delta>0$:
\begin{proposition}\label{prop2}
Let $\psi:\R^n\to \R$ be a convex function satisfying \eqref{eq:O}, 
\eqref{eq:ineq1}, \eqref{eq:ineq2}, with $\mu$ and $\nu$ satisfying \eqref{eq:munu1} and \eqref{eq:munu2} with $\delta>0$.
There exists some constants $c$, $C_1$, $C_2$ depending on $n$, $\lambda$,  $\Lambda $, $\Omega$ and $\pa\psi(\Omega)$ such that the following holds:
Given $x_0\in \Omega$, $p_0\in \pa\psi(x_0)$, $t$ such that $S(x_0,p_0,C_1 t)\subset\Omega $ and if 
$$t \geq C_0 \delta \quad\mbox{  and }\quad  \ell(S) \geq r \geq \delta^{1/2}$$ 
and 
\begin{equation}\label{eq:tC2}
 \frac{t}{\ell(S)^2} \geq C_2
\end{equation}
then 
there exists a set $  \widetilde \Sigma \subset C S(x_0,p_0, t)  $ 
such that 
$$ 
|\widetilde \Sigma | \geq c|S|$$
and 
$$ 
\Delta_r \psi (x) \geq c  \frac{t}{\ell(S)^2}  \quad \mbox{ a.e. in $\widetilde \Sigma$}.
$$
\end{proposition}
The assumptions require a short discussion:
The condition $t \geq C_0 \delta$ is necessary to use the standard properties of section Proposition \ref{prop:sec}.
The condition $r<\ell(S)$ is necessary to control the nonlocal quantity $\Delta_r \psi$. The condition $r \geq \delta^{1/2}$ will then be used  to ensure that the sets $\tilde \Sigma$  and $\tilde V$ are "fat" enough to be seen by the (possibly discrete) measures $\mu$ and $\nu$.
Finally, we note that \eqref{eq:ellt} implies  $ \frac{t}{\ell(S)^2} \geq c$ for some constant $c$, but  \eqref{eq:tC2} requires this quantity to be arbitrarily large. 
Instead of \eqref{eq:tC2}, we could actually assume that $r \geq C_3 \delta^{1/2}$ for some large constant $C_3$. 
However, we will only use this proposition for sections for which $ \frac{t}{\ell(S)^2} $ is very large (see \eqref{eq:tx}), so this condition will not be an issue.

\begin{proof}
In view of \eqref{eq:SV} and \eqref{eq:SW}, Assumption \ref{ass:1} and Proposition \ref{prop:delta} imply
$$ \mu(\widetilde \Sigma)=\mu(\widetilde \Sigma\cap\Omega) \leq \Lambda |\widetilde \Sigma|\quad\mbox{ and }\quad \nu(\widetilde V\cap\O) \geq \lambda |\widetilde V\cap \O| $$
so that \eqref{eq:ineq2} and Proposition \ref{prop:Vh} (i) (which gives $\pa\psi^*(\widetilde V)\subset \widetilde\Sigma$) imply
$$\lambda |\widetilde V\cap \O| \leq \nu(\widetilde V \cap \O) \leq \mu(\pa\psi^*(\widetilde V\cap \O))  \leq \mu(\pa\psi^*(\widetilde V))  \leq \mu(\widetilde\Sigma) 
 \leq \Lambda |\widetilde \Sigma|.
 $$

Using \eqref{eq:papsij}, and the fact that $V\subset\pa\psi(S)\subset \O$, we deduce 
$$c\frac{t^n}{|S|} \leq |V| \leq |\widetilde V\cap \O| \leq \frac{\Lambda}{\lambda} |\widetilde \Sigma|.
 $$
Finally, Proposition \ref{prop:sec} (i)  gives $t^n \geq c |S|^2$ which yields $  |\widetilde \Sigma |\geq c|S|$ for some constant $c$. 
 \end{proof}

\subsection{A reversed Chebychev's inequality when $\delta>0$}
As in the case $\delta=0$, Theorem \ref{thm:2} follows from the following reversed Chebychev's inequality:
\begin{theorem}\label{thm:chebyd}
Let $\psi:\R^n\to \R$ be a convex function satisfying 
\eqref{eq:ineq1}, \eqref{eq:ineq2}, with $\mu$ and $\nu$ satisfying \eqref{eq:munu1} and \eqref{eq:munu2}) with $\delta>0$.
Let $\Omega' \subset\subset \Omega$ be such that 
there exists $\rho>0$  such that \eqref{eq:rho1} holds.
Then there exists $\alpha_0$, $r_0$, $c$ and $C$  (depending on 
$\Omega$, $\rho$, $n$, $\lambda$, $\Lambda$, $L(\pa\psi(\Omega))$)
such that whenever
 $$\alpha>\alpha_0,\qquad  r<r_0,\qquad  \delta\leq r^2,$$
 the following holds:
\begin{equation}\label{eq:chebyd}
\int_{\{x\in \Omega'\, |\, \Delta_r \psi \geq \alpha \}} \Delta_r \psi  \, dx \leq C \alpha |  \{x\in \Omega\,|\,  \Delta_r \psi  \geq c\alpha\}|.
\end{equation}
\end{theorem}

Before proving this result, we show that it implies our main Theorem \ref{thm:2}:
\begin{proof}[Proof of Theorem \ref{thm:2}]
Since $|\Delta_r \psi|\leq \frac{C}{r_0^2} \sup\psi$ when $r\geq r_0$, we only need to prove the result when $r<r_0$. 
The proof is identical to the proof of Theorem \ref{thm:1}: Using Theorem \ref{thm:chebyd} and Fubini's theorem we can show that there exists a constant $C$ such that
$$
\int_{\Omega'} \Delta_r\psi(x)\log(2+\Delta_r\psi(x))\, dx
\leq C + C \int_{  \Omega''} \Delta_r\psi(x)\, dx,
$$
where $\Omega\subset\subset\Omega''\subset\subset \Omega$ and we can further have that $\Omega''+B(0,r)\subset \Omega$ by taking $r_0$ small enough with respect to $\rho$.

We can now observe that
\[
\Delta_r \psi(x)=L_r\star \Delta \psi(x),
\]
where $\Delta\psi$ is a non-negative Radon measure and $L_r$ is a smooth convolution kernel, bounded in $L^1$ and obtained through averages of $K_r$.

Consequently,
\[
\int_{  \Omega''} \Delta_r\psi(x)\, dx=\int_{  \Omega''} L_r\star\Delta\psi(x)\, dx\leq \int_{\Omega} \Delta \psi\leq C.
\]
\end{proof}

We now turn to the proof of Theorem  \ref{thm:chebyd}.

\begin{proof}[Proof of Theorem \ref{thm:chebyd}]
The proof follows the same idea as that of Theorem \ref{thm:cheby} but requires several adjustments: Given the set 
$$ E_\alpha = \{x\in\Omega'\, |\, \Delta_r \psi (x)>\alpha\} ,$$  
we need to identify, for each $x\in E_\alpha$ the value $t(x)$ of $t$ with the following properties:
\begin{enumerate}
\item We should have $t(x)\geq C_0\delta$ so that we are only working with sections for which the properties listed in Proposition \ref{prop:sec} are valid.
\item We need to have $\fint_{S(x,t)} \Delta_r \psi \, dy \leq C\alpha$ for all $t\in[t(x),\rho]$
\item We need to have $\frac{t(x)}{\ell(S(x,t(x))^2}\geq c \alpha$ in order to use Proposition \ref{prop2}.
\item The particular choice of $q\in \pa\psi(x)$ to define the section is not important except for the fact that we need to take $q$ such that \eqref{eq:rho1} holds.
\end{enumerate}
Lemma \ref{lem:supd} can be used to show that $\fint_{S(x,q,t)} \Delta_r \psi\, dy \leq C \frac{t}{\ell(S(x,q,t(x))^2} $ if $t\geq C_0\delta$ and $r\leq \ell(S(x,q,t(x))$. This allows us to replace items 2 and 3 
with the condition
$$
 \frac{t}{\ell(S(x,q,t(x))^2} \leq C\alpha \quad \forall t\in[t(x),\rho], \qquad \frac{t(x)}{\ell(S(x,q,t(x))^2}\geq c \alpha.
$$
Making sure that $t(x)\geq C_0\delta$ is more delicate, but we note that if $B_r(x)\subset S(x,q,t)$, then the definition of $\Delta_r$ implies $\Delta_r\psi(x)\leq m_0 \frac{t}{r^2}$. For $x\in E_\alpha$, we thus have $t \geq \frac{\alpha}{m_0} r^2 \geq  \frac{\alpha_0}{m_0} \delta$ and we get the desired inequality if $\alpha_0$ is large enough. It turns that requiring that $ B_r(x)\subset S(x,q,t(x))$ is too strong so we will require 
$$ B_r(x)\subset S(x,q,M t(x))$$
where $ M$ is the constant given in Proposition \ref{prop:sec} such that $2S(x,q,t)\subset S(x,q,Mt)$.

All these considerations lead us to introduce for all $x\in E_\alpha$,
\begin{equation}\label{eq:txdef} 
t(x) = \sup \left\{ t \leq \rho\,\Big|\, \max\left\{ \frac{t}{ \ell(S(x,q,t)) ^2},\frac{r\alpha}{d(x,\pa S(x,q,M t)) }\right\} \geq \alpha\right\}.
\end{equation}
Given  $\eps=\eps_1\eps_2>0$ (to be chosen later) we claim that if $\alpha$ is large enough and $r$ is small enough, then we must have $t(x)\leq \eps\rho$.
Indeed, we recall (see Proposition \ref{prop:lL}) that $\ell(S(x,q,t)) \geq \frac{t}{L(\pa\psi(\Omega))}$ and so
$$ \frac{t}{ \ell(S(x,q,t)) ^2}  \leq \frac{L(\pa\psi(\Omega))^2}{t}\leq\alpha_0 := \frac{L(\pa\psi(\Omega))^2}{\eps\rho }\qquad \forall x\in \Omega', \; \forall t\in[\eps\rho,\rho]$$
while Alexandrov's estimate (Proposition \ref{prop:alek}) implies 
$$ d(x,\pa S(x,q,M t)) \geq r_0 := \frac{(\eps \rho)^n}{|\pa\psi(\Omega)| L(\Omega)^{n-1}}\qquad  \forall x\in \Omega', \; \forall t\in[\eps\rho,\rho].$$
From now on, we assume that $r<r_0$ and $\alpha>\alpha_0$. We then have, for all $x\in \Omega'$, 
$$
 \max\left\{ \frac{t}{ \ell(S(x,q,t)) ^2},\frac{r\alpha}{d(x,\pa S(x,q,Mt)) }\right\} < \alpha\qquad \forall t\in[\eps\rho,\rho]
$$
and so $t(x)\leq \eps\rho$. 

The definition of $t(x)$, \eqref{eq:txdef}  implies
 \begin{equation} \label{eq:tlnf}
 \frac{t}{ \ell(S(x,q,t)) ^2} \leq \alpha \qquad\forall x\in E_\alpha, \;  \forall t\in [t(x),\rho]
 \end{equation}
and 
\begin{equation}\label{eq:ltx}
d(x,\pa S(x,q,Mt))\geq r \qquad \qquad\forall x\in E_\alpha, \;  \forall t\in [t(x),\rho].
 \end{equation}
The first condition was one of the required conditions for $t(x)$ and the second condition can be used to show that $t(x)\geq C_0\delta$.
Indeed, for $x\in E_\alpha$, condition \eqref{eq:ltx} implies that $B_r(x)\subset S(x,q,Mt)$ and so
\begin{align}
\alpha < \Delta_r\psi (x)
& = \frac{m_0}{r^2}   \int_{B_r(x)} K_r(y-x) [\psi(y) -\psi(x) - q\cdot (y-x)]\, dy\nonumber  \\
& \leq  \frac{m_0}{r^2}M t   \int  K_r(y) \, dy  =  \frac{m_0M}{r^2} t\qquad  \forall t\in [t(x),\rho] \label{eq:jfh}
\end{align}
and so (with $t=t(x)$)
\begin{equation}\label{eq:txbd}
t(x) \geq \frac{\alpha}{m_0M} r^2 \geq  \frac{\alpha_0}{m_0M}  \delta.
\end{equation} 
We can always take a larger $\alpha_0$ so that $\frac{\alpha_0}{m_0M} >C_0$ (which we assume in what follows).

Finally, we can prove:
\begin{lemma}\label{lem:tla}
Let $t(x)$ be defined by \eqref{eq:txdef}. Then
\begin{equation}\label{eq:tx}
 \frac{t(x)}{ \ell(S(x,q,t(x))) ^2} \geq c \alpha,   \qquad \forall x\in E_\alpha
\end{equation}
and 
\begin{equation}\label{eq:intalpha}
 \fint_ {S(x,q,t)}   \Delta_r  \psi (y) \, dy \leq C \alpha \qquad \forall x\in E_\alpha, \quad  t\in [t(x),\eps_1\rho].
\end{equation}
\end{lemma}
\begin{proof}[Proof of Lemma \ref{lem:tla}]
  Observe that for a fixed $x$ and $q\in \pa\psi(x)$, the map $t\to S(x,q,t)$ is continuous thanks to the convexity of $\psi$. Indeed considering a sequence $t_n\to t$ with for example $t_n<t$, we first trivially have that $\bigcup_n S(x,q,t_n)\subset S(x,q,t)$ and thus $\overline{\bigcup_n S(x,q,t_n)}\subset S(x,q,t)$. On the other hand if there exists $y\in S(x,q,t)\setminus \overline{\bigcup_n S(x,q,t_n)}$ then necessarily $q\in \pa\psi(y)$. As $q\in \pa\psi(x)$, by the convexity of $\psi$, this is possible only if $t=\psi(x)$.
  
  The definition \eqref{eq:txdef} then implies that
$$
\max\left\{ \frac{t(x)}{ \ell(S(x,q,t(x))) ^2},\frac{r\alpha}{d(x,\pa S(x,q,Mt(x)))}\right\} = \alpha.$$
In particular we either have
$ \frac{t(x)}{ \ell(S(x,q,t(x))) ^2} = \alpha$ (and \eqref{eq:tx} is true)
or $ \frac{t(x)}{ \ell(S(x,q,t(x))) ^2} < \alpha$ and $d(x,\pa S(x,q,Mt(x)))=r$.
In that last case we go back to \eqref{eq:jfh} which implies
$$\alpha \leq  m_0 M  \frac{t(x)}{r^2}  =  m_0 M \frac{t(x)}{d(x,\pa S(x,q,Mt))^2} .$$
In order to conclude, we note that (here we write $\ell = \ell(S(x,q,t(x)))$)  $B_{\ell} (0) \subset S(x,q,t(x))$ and so (using the convexity of $S$) $B_{\ell} (x) \subset 2S(x,q,t(x)) \subset S(x,q,Mt(x))$ (using Proposition \ref{prop:sec}  and \eqref{eq:txbd}). It follows that 
$\ell(S(x,q,t(x)) \leq Cd(x,\pa S(x,q,Mt))$ and the inequality above implies
$$\alpha \leq   m_0 M \frac{t(x)}{\ell(S(x,q,t(x))) ^2} $$
which is \eqref{eq:tx}.

To prove inequality  \eqref{eq:intalpha},
we take $t\in [t(x),\eps_1\rho]$ and $x\in E_\alpha$.
Lemma \ref{lem:supd}, together with the fact that   $\mathrm{Tr}(A)\leq \frac{n}{\ell(S)^2}$ and $\ell(S(x,q,t)) >  r$ implies
$$
 \fint_ {S(x,q,t)}   \Delta_r  \psi (y) \, dy \leq \frac{C}{\ell(S(x,q,t))^2} \sup_{y\in 3 n S(x,q,t)} [\psi (y)-\psi(x)-q\cdot (y-x) ] .
$$
Condition \eqref{eq:txbd}  implies that the scaling property of section \eqref{eq:secd} can be used to prove that there exists a constant $C_1=(3n)^\beta$ such that $3 n S(x,q,t)\subset S(x,q,C_1t)$ provided $S(x,q,C_1 t) \subset \Omega$ which holds if we choose 
$\eps_1 = \frac{1}{C_1}$ (since $t\leq \eps_1 \rho$).
We
deduce
$$
 \fint_ {S(x,q,t)}   \Delta_r  \psi (y) \, dy \leq \frac{C t }{\ell(S(x,q,t))^2}  .
$$
Inequality \eqref{eq:tlnf} then implies \eqref{eq:intalpha}.
\end{proof}

\medskip

The rest of the proof is similar to the case $\delta=0$:
We have $E_\alpha \subset \bigcup_{x\in E_\alpha} S(x,q,t(x))$ and
since $t(x) \geq C\delta$, we can use Vitali's covering lemma (see Corollary \ref{cor:vitali}) to extract a countable family $x_k$ such that the sections $S_k = S(x_k,q_k,t(x_k))$ are pairwise disjoint and  $E_\alpha \subset \bigcup_{k}  C_{*} S_k$.\medskip

Successive applications of Proposition \ref{prop:sec1}
gives
$ C_{*} S_k \subset S(x_k,q_k,C_{**}   t_k) $ for some constant $C_{**}$
provided 
$C_{**}   t _k\leq \rho.$
Furthermore, in order to be able to use Inequality \eqref{eq:intalpha}, we will require that $C_{**}   t _k\leq \eps_1 \rho$.
We thus choose 
$$\eps_2: = \frac1 { C_{**} } $$
The condition $t_k\leq \eps \rho$ with $\eps=\eps_1\eps_2$ then imply that $ C_{**}  t_k\in[ t_k,\eps_1\rho]$. Inequality \eqref{eq:intalpha}
 thus implies
$$  \int_{C_{*} S_k}  \Delta_r  \psi (x) \,dx  \leq \int_{S(x_k,q_k,C_{**}  t_k)}  \Delta_r  \psi (x) \,dx \leq \alpha | S(x_k,q_k,C_{**}   t_k)|$$
and Proposition \ref{prop:sec} (i) gives $| S(x_k,q_k,C_{**}   t_k)| \leq C_n |S(x_k,q_k,t_k)|$. We deduce that
$$
 \int_{C_{*} S_k} \Delta_r  \psi (x)  \,dx\leq C\alpha |S_k|.
 $$
We are ready to conclude: First, we write
\begin{align*}
\int_{E_\alpha} \Delta_r  \psi (x)  \, dx
& \leq \sum_k \int_{  C_{*} S_k} \Delta_r  \psi (x)  \, dx \leq C\alpha  \sum_k | S_k| .
\end{align*}
We then use Proposition \ref{prop2}:  We have
$\frac{t_k}{ \ell(S_k) ^2} \geq c\alpha$ (see \eqref{eq:tx}), so
 there is a subset $\Sigma_k\subset S_k$ such that 
$$|\Sigma_k|\geq c|S_k|, \qquad  
\Delta_r \psi(x) \geq c\frac{t_k}{ \ell(S_k) ^2} \geq c\alpha \quad\forall x\in\Sigma_k,$$ 
and since  the $S_k$, and therefore the $\Sigma_k$, are disjoint, we have that
\begin{align*}
\int_{E_\alpha} \Delta_r  \psi (x)  \, dx
& \leq C\alpha  \sum_k | \Sigma_k|  \\
& \leq C\alpha|  \{ x\in \Omega \,|\, \Delta_r  \psi (x)  \geq c\alpha\}|
\end{align*}
which is the desired inequality.
\end{proof}
\subsection{Sobolev regularity \label{sobolev}}
A key step in deriving Sobolev $W^{2,p}$ estimates on $\psi$ up to scale $r$ consists in localizing the reversed Chebychev's inequality of Theorem~\ref{thm:chebyd}.
\begin{theorem}\label{thm:chebydloc}
Let $\psi:\R^n\to \R$ be a convex function satisfying 
\eqref{eq:ineq1}, \eqref{eq:ineq2}, with $\mu$ and $\nu$ satisfying \eqref{eq:munu1} and \eqref{eq:munu2}) with $\delta>0$.
Let $\Omega' \subset\subset \Omega$ be such that 
there exists $\rho>0$  such that \eqref{eq:rho1} holds.
Then there exists $\alpha_0$, $r_0$, $c$ and $C$  (depending on 
$\Omega$, $\rho$, $n$, $\lambda$, $\Lambda$, $L(\pa\psi(\Omega))$)
such that whenever
 $$\alpha>\alpha_0,\qquad  r<r_0,\quad \alpha < \frac{c}{r^{1/n}}, \qquad  \delta\leq r^2,$$
 the following holds for any ball $B(x_0,R)\subset \Omega'$
\begin{equation}\label{eq:chebydloc}
\int_{\{x\in B(x_0,R)\, |\, \Delta_r \psi \geq \alpha \}} \Delta_r \psi  \, dx \leq C \alpha |  \{x\in B(x_0,R+C\,\alpha^{-1/\beta})\,|\,  \Delta_r \psi  \geq c\alpha\}|.
\end{equation}
\end{theorem}
\begin{proof}
  The proof follows the exact same lines as the proof of Theorem~\ref{thm:chebyd} by covering the set $\{x\in B(x_0,R)\, |\, \Delta_r \psi \geq \alpha \}$ with sections $S(x,t(x))$ where $x\in B(x_0,R)$. This leads to
  \[
  \int_{\{x\in B(x_0,R)\, |\, \Delta_r \psi \geq \alpha \}} \Delta_r \psi  \, dx \leq C \alpha |  \{x\in B(x_0,R+R')\,|\,  \Delta_r \psi  \geq c\alpha\}|,
  \]
  where $R'=\sup_{x\in B(x_0,R)} L(S(x,t(x))$.

  The only difference with Theorem~\ref{thm:chebyd} therefore lies in bounding precisely $R'$. We go back to the definition~\eqref{eq:txdef} of $t(x)$ and use first again Proposition~\ref{prop:lL} which gives $\ell(S(x,q,t))\geq \frac{t}{L(\partial\psi(\Omega))}$, so that for some constant $C$ we have
  \[
\frac{t}{\ell(S(x,q,t))^2}\leq \frac{L(\partial\psi(\Omega))^2}{t}<\alpha,\quad\forall t> \frac{C}{\alpha}.
\]
Similarly using Proposition~\ref{prop:alek}, we have that
\[
d(x,\pa S(x,q,Mt))\geq \frac{t^n}{|\partial\psi(\Omega)|\,L(\Omega)^{n-1}}> r,\quad \forall t> C\,r^{1/n}.
\]
Since we assumed that $\alpha < \frac{c}{r^{1/n}}$, this last inequality holds in particular for all $t> \frac{C}{\alpha}$.
Taken together, the two inequalities and \eqref{eq:txdef} imply that  $t(x)\leq \frac{C}{\alpha}$.

To conclude, it is now enough to use Lemma~\ref{lem:diam} which shows that
\[
L(S(x,p,t))\leq C\,t(x)^{1/\beta}\leq \frac{C}{\alpha^{1/\beta}}.
\]
Since this bound holds for all $x\in B(x_0,R)$, we find $R' \leq C \alpha^{-1/\beta}$, which concludes the proof.
\end{proof}
Using Theorem~\ref{thm:chebydloc}, we can now prove the Sobolev regularity Theorem \ref{thm:sobolev}
\begin{proof}[Proof of Theorem \ref{thm:sobolev}]
The proof mostly follows the arguments presented in \cite{DF} with only minor adaptations.
Since $|\Delta_r \psi|\leq \frac{C}{r_0^2} \sup\psi$ when $r\geq r_0$, we only need to prove the result when $r<r_0$.

For some $\gamma>1$ to be chosen later, we define $\alpha_k$ by induction with $\alpha_{k+1}=\gamma\,\alpha_k$, $\alpha_1$ large enough, and $R_{k+1}=R_k-C/\alpha_k^{1/\beta}$ with $R_1=2\,R$.

As long as $\alpha_{k}\leq \frac{C}{r^{1/n}}$ and provided that $\gamma$ was chosen large enough so that $R_k\geq R$ for all $k$, we can apply Theorem~\ref{thm:chebydloc} to have that
  \[
  \int_{\{x\in B(x_0,R_{k+1})\, |\, \Delta_r \psi \geq \alpha_k/c \}} \Delta_r \psi  \, dx \leq \frac{C}{c} \alpha_k |  \{x\in B(x_0,R_k)\,|\,  \Delta_r \psi  \geq \alpha_k\}|.
  \]
  As a consequence
  \[\begin{split}
  \int_{\{x\in B(x_0,R_{k+1})\, |\, \Delta_r \psi \geq \alpha_{k+1} \}} \Delta_r \psi  \, dx \leq & \frac{C}{c} \alpha_k |  \{x\in B(x_0,R_k)\,|\,  \alpha_k\leq \Delta_r \psi  \leq \alpha_{k+1}\}|\\
  &+\frac{C}{c\,\gamma}\,\alpha_{k+1}\,|  \{x\in B(x_0,R_k)\,|\,  \Delta_r \psi  \geq \alpha_{k+1}\}|,
  \end{split}
  \]
  or, taking $\gamma\geq 2C/c$,
 \[\begin{split}
  \int_{\{x\in B(x_0,R_{k+1})\, |\, \Delta_r \psi \geq \alpha_{k+1} \}} \Delta_r \psi  \, dx \leq & \frac{C}{c}\,\int_{\{x\in B(x_0,R_{k})\, |\, \alpha_k\leq \Delta_r \psi \leq \alpha_{k+1} \}} \Delta_r \psi  \, dx \\
  &+\frac 1 2 \,\int_{\{x\in B(x_0,R_{k}))\, |\, \Delta_r \psi \geq \alpha_{k+1} \}} \Delta_r \psi  \, dx.
  \end{split}
  \]
We deduce:
 \[\begin{split}
  \int_{\{x\in B(x_0,R_{k+1})\, |\, \Delta_r \psi \geq \alpha_{k+1} \}} \Delta_r \psi  \, dx \leq & \frac{2C}{c}\,\int_{\{x\in B(x_0,R_{k})\, |\, \alpha_k\leq \Delta_r \psi \leq \alpha_{k+1} \}} \Delta_r \psi  \, dx \\
  &+ \,\int_{\{x\in B(x_0,R_{k})\setminus B(x_0,R_{k+1})\, |\, \Delta_r \psi \geq \alpha_{k+1} \}} \Delta_r \psi  \, dx.
  \end{split}
  \]
  Defining correspondingly $D_k=\{x\in B(x_0,R_{k})\, |\, \Delta_r \psi \geq \alpha_{k} \}$, this is equivalent, for a different constant $C$, to
  \[
\int_{D_{k+1}} \Delta_r \psi  \, dx\leq C\, \int_{D_{k}\setminus D_{k+1}} \Delta_r \psi  \, dx,
  \]
  or
  \[
(1+C)\,\int_{D_{k+1}} \Delta_r \psi  \, dx\leq C\, \int_{D_{k}} \Delta_r \psi  \, dx.
  \]
  This implies that
  \[
  \int_{D_{k+1}} \Delta_r \psi  \, dx\leq \tau\,\int_{D_{k}} \Delta_r \psi  \, dx,
  \]
  for $\tau=C/(1+C)<1$.

  By induction, and since $B(x_0,R)\subset B(x_0,R_k)$, we can immediately deduce that
  \begin{equation}
\int_{\{x\in B(x_0,R)\, |\, \Delta_r \psi \geq \alpha_{k} \}} \Delta_r \psi  \, dx\leq \tau^k\int_{\Omega} \Delta_r\psi\,dx\leq \frac{C}{\alpha_k^\theta},\label{inductiontau}
\end{equation}
as long as $\alpha_k\leq \frac{C}{r^{1/n}}$ and for some $\theta>0$ since $\alpha_k=\gamma^k\,\alpha_1$ (so that we can take $\theta=-\log \tau/\log \gamma$).

Denote $k_0$ the last index $k$ s.t. $\alpha_{k_0}\leq  \frac{C}{r^{1/n}}$ and decompose
\[
\begin{split}
\int_{B(x_0,R)} (\Delta_r \psi)^p\,dx=\int_{B(x_0,R)\cap\{\Delta_r\psi\leq \alpha_{k_0}\}} (\Delta_r \psi)^p\,dx+\int_{B(x_0,R)\cap\{\Delta_r\psi> \alpha_{k_0}\}} (\Delta_r \psi)^p\,dx.
\end{split}
\]
Because $\Delta_r \psi(x)\leq \frac{C}{r^2}$, the last term is in fact easy to handle,
\[\begin{split}
\int_{B(x_0,R)\cap\{\Delta_r\psi> \alpha_{k_0}\}} (\Delta_r \psi)^p\,dx&\leq \frac{C}{r^{2(p-1)}}\,\int_{B(x_0,R)\cap\{\Delta_r\psi> \alpha_{k_0}\}} \Delta_r \psi\,dx\\
&\leq \frac{C}{r^{2(p-1)}}\,\frac{1}{\alpha_{k_0}^\theta}\leq C\,\frac{r^{\theta/n}}{r^{2(p-1)}}\leq C,
\end{split}
\]
provided that $p$ is close enough to $1$ so that $\theta/n\geq 2(p-1)$.

We now recall that for any measurable function $f(x)\geq 0$ and $p>1$,
\[
\int_{B(x_0,R)} f^p\,dx=(p-1)\,\int_0^\infty \alpha^{p-2}\,\int_{\{x\in B(x_0,R)\, |\, f(x) \geq \alpha\}} f  \, dx.
\]
Hence
\[
\begin{split}
  \int_{B(x_0,R)\cap\{\Delta_r\psi\leq \alpha_{k_0}\}} (\Delta_r \psi)^p\,dx&=(p-1)\,\int_0^{\alpha_{k_0}} \alpha^{p-2}\,\int_{\{x\in B(x_0,R)\, |\, \Delta_r \psi \geq \alpha\}} \Delta_r \psi \, dx\\
  &\leq C+C\,\sum_{k=1}^{k_0} \alpha_{k}^{p-1}\,\int_{\{x\in B(x_0,R)\, |\, \Delta_r \psi \geq \alpha_k\}} \Delta_r \psi \, dx.
\end{split}
\]
By using again~\eqref{inductiontau}, we finally conclude that
\[
 \int_{B(x_0,R)\cap\{\Delta_r\psi\leq \alpha_{k_0}\}} (\Delta_r \psi)^p\,dx\leq C+C\,\sum_{k=1}^{k_0} \gamma^{k(p-1-\theta)}\leq C,
 \]
 as long as $p$ is again selected close enough to $1$ s.t. $p-1-\theta<0$.
\end{proof}
\section{Properties of sections}\label{sec:sec}
In this section, we establish some important properties of the sections of  a convex function that satisfies inequalities \eqref{eq:ineq1}-\eqref{eq:ineq2} for measures $\mu$ and $\nu$ such that Assumption \ref{ass:1} (inequalities \eqref{eq:munu1}-\eqref{eq:munu2}) holds.
Namely we establish some bounds on the volume of the sections (Proposition \ref{prop:St}), a scaling property of sections (Proposition \ref{prop:sec1}) and the engulfing property of sections (Proposition \ref{prop:engulf}).
These three  properties are classical when $\delta=0$ (see for instance
\cite{Caffarelli91},
\cite{Gutierrez00},
\cite[Chapter 3]{Gutierrez01}) and play an important role in the regularity theory for Monge-Amp\`ere equation.
When $\delta>0$, it is clear that they can only hold for large enough sections. The main interest of this section is thus to clearly identify the condition on $t$ (depending on $\delta$) for these conditions to hold.

An important tool in the proof of these results when $\delta=0$  is  the renormalization of convex sets provided by John's lemma. However, this renormalization procedure is not compatible with Assumption \ref{ass:1}
  since it changes the length scales differently depending on the direction. 
For this reason, our proofs will avoid the use of renormalization whenever Assumption~\ref{ass:1} must be used. They will make extensive use of the notion of polar body of a convex set, a notion well suited to optimal transportation. To our knowledge this approach is new even in the case $\delta=0$.
 
Throughout this section, we will assume that $\mu$ and $\nu$ satisfy the following condition, which is equivalent to Assumption \ref{ass:1} (see Proposition \ref{prop:bhkh}):
\begin{assumption}\label{ass:2}
There exist $\lambda$, $\Lambda>0$ such that
\begin{equation}\label{eq:mn3}
\lambda |K| \leq \mu(K) \leq\Lambda |K|, \quad \mbox{for all convex set $K\subset \Omega$ with $\ell(K) \geq  \delta$} 
\end{equation}
and 
\begin{equation}\label{eq:mn4}
\lambda |K| \leq \nu(K) \leq\Lambda |K|, \quad \mbox{for all convex set $K\subset \O$ with $\ell(K) \geq  \delta$.} 
\end{equation}
\end{assumption}
\subsection{Polar bodies of a convex set}
Several proofs in this section will make use of the notion of polar body of a convex set, which turns out to be a convenient tool to relate the section $S$ of a convex function $\psi$ and its image $\pa \psi (S)$.
For a convex set $K$ containing the origin, we recall that the polar body is defined  by
\begin{equation}\label{def:polar} 
K^\circ = \{ z\in \R^n\, ;\, z\cdot x \leq 1, \quad \forall x\in K\}.
\end{equation}
We recall the following basic facts about polar bodies:
\begin{enumerate}
\item If $A\subset B$ then $B^\circ \subset A^\circ$
\item $B_r(0)^\circ = B_{1/r}(0)$
\item The volume product $|K| \, |K^\circ|$ is unchanged by linear (but not affine) transformations: Indeed, 
$$ z\in (L(K))^\circ \Leftrightarrow z\cdot L(x) \leq 1 \; \forall x\in K  \Leftrightarrow L^*( z)\in K ^\circ$$
and so 
$|L(K)|\, |(L(K))^\circ| = |\det L| |K|\, |\det (L^*)^{-1}| |K^\circ| =  |K|\,  |K^\circ|$.
\end{enumerate}

There is an abundant literature devoted to the study of the polar body and in particular to the derivation of optimal lower and upper bounds for the volume product $|K| \, |K^\circ|$. We will use the following classical result:
\begin{proposition}\label{prop:polarinf}
Let $K$ be a bounded convex subset of $\R^n$ with non-empty interior.
We have:
\begin{equation}\label{eq:polarsup1} 
\inf_{x_0\in K} |K| \, |(K-x_0)^\circ | \leq n^n\omega_n^2
\end{equation}
(the unique point $x_0$ where the infimum in \eqref{eq:polarsup1} is attained is called the Santal\'o point of $K$)
and 
\begin{equation}\label{eq:polarsub1}  
|K| \, |K^\circ| \geq \frac{ \omega_n^2}{(2n)^n}
\end{equation}
where $\omega_n$ denotes the volume of the $n$-dimensional unit ball.
\end{proposition}

These inequalities are classical and we give a short proof below for the sake of completeness.
Note that the constants in   \eqref{eq:polarsup1} and \eqref{eq:polarsub1} are not optimal. 
We refer for example to Bourgain-Milman \cite{BM} for discussion about optimal constants.

\begin{proof}
When $K$ is a convex body that is symmetric with respect to the origin, John's Lemma \ref{lem:J} implies that there is a linear transformation $L$ such that  $B_1(0)\subset L(K)\subset B_{n}(0)$. We then have $B_{1/n}(0)\subset (L(K))^\circ \subset B_{1}(0)$ and so 
\begin{equation}\label{eq:SSo}
\frac{\omega_n^2}{n^n}\leq |L(K)|\, |(L(K))^\circ| = |K|\,  |K^\circ| \leq n^n\omega_n^2.
\end{equation}

When $K$ is not symmetric with respect to the origin, John's Lemma gives $B_1(a)\subset L(K)\subset B_{n}(a)$ for some $a=L(y)$ with $y\in K$. 
The same computation as above, applied to $K-y$ thus gives
$$\frac{\omega_n^2}{n^n}\leq  |K|\,  |(K-y)^\circ| \leq n^n\omega_n^2$$
which gives in particular \eqref{eq:polarsup1}.
 
We note that if $0\notin K$, then $|K^\circ|$ is infinite and \eqref{eq:polarsub1}   holds. So we can assume that $0\in K$ and thus $0\in L(K)\subset B_{n}(a)$. It follows that   $L(K) \subset  B_{2n}(0)$ and so $L(K)^\circ \supset  B_{1/2n}(0)$. Proceeding as above, we deduce
$$ |K|\,  |K^\circ|= |L(K)|\, |(L(K))^\circ| \geq |B_1| | B_{1/2n}|=
\frac{\omega_n^2}{(2n)^n}$$
which proves \eqref{eq:polarsub1}.  
\end{proof}

Finally, we can get the following bounds on the inner and outer radii of $K$ and $K^\circ$.
\begin{lemma}\label{lem:radK}
Let $K$ be a  bounded convex set with non-empty interior such that $0\in K$. Then
$$\ell(K) L(K^\circ) \geq \frac{1}{2n} \qquad \ell(K^\circ)L(K) \geq \frac{1}{2n }
$$
\end{lemma}

\begin{proof}
John's Lemma implies the existence of an ellipsoid $E$ such that
$$ E\subset K \subset nE,$$
where we note that $nE$ denotes a dilation of $E$ about its center, which is not necessarily $0$.
In particular, we have $\ell(K)\geq \ell(E)$ and 
$ (nE)^{\circ} \subset K^\circ$.

Next, we note that the polar dual of a centered ellipsoid $E_0=\left\{ \sum_{i=1}^n \left(\frac{x_i}{a_i}\right)^2 \leq 1\right\}$ is the ellipsoid 
$E_0^\circ =\left\{ \sum_{i=1}^n\left(a_iu_i \right)^2 
 \leq 1\right\}$. In particular 
 $\ell(E_0^\circ) L(E_0)=1$ and $L(E_0^\circ)  \ell(E_0)=1$ for such an ellipsoid.

The ellipsoid $F:=nE$ might not be centered, but we have $0\in F $, so there is a centered ellipsoid $F_0$, twice as large as $F$, such that $F\subset F_0$ and $F_0^\circ \subset F^\circ$. We have
 $L(F^\circ) \geq L(F_0^\circ) =  \frac{1}{\ell(F_0)}=  \frac{1}{2\ell(F)}$  and $L(F) = \frac 1 2 L(F_0) = \frac  1{2\ell(F_0^\circ)}\geq  \frac 1 {2\ell(F^\circ)}$. We thus have:
 $$ \ell(nE) L((nE)^\circ) \geq \frac 1 2, \qquad  L(nE) \ell((nE)^\circ) \geq \frac 1 2.$$ 

Finally, we can write
$$\ell(K)\geq \ell(E) = \frac{1}{n} \ell(nE) \geq  \frac{1}{2 n L((nE)^\circ)}  \geq  \frac{1}{2 n L(K^\circ)}  $$
which gives the first inequality in Lemma \ref{lem:radK}.
 $K^\circ$ is a convex set containing zero and $(K^\circ)^\circ=K$ so the second inequality follows by applying the first inequality to $K^\circ$.
 \end{proof}
  
 This lemma readily implies. 
\begin{corollary}\label{cor:IL}
Let $\Omega$ be a bounded subset of $\R^n$ and $\psi:\Omega\to\RR$ convex. 
For any section $S \subset \Omega$, we have 
$$  \ell((S-y)^\circ) \geq \frac{1}{2n L(\Omega)} \qquad \forall y\in S.$$
\end{corollary}
\begin{proof}
The result follows from  Lemma \ref{lem:radK} since we have 
$$ \frac 1 {2n} \leq \ell((S-y)^\circ)  L (S-y) \leq \ell((S-y)^\circ)  L (\Omega).$$
\end{proof}
The usefulness of  polar bodies in the context of optimal transportation is made clear by the next proposition.
We recall that for a convex function $\psi:\R^n\to \R$  and 
given $x_0\in\R^n$, $p_0\in \pa\psi(x_0)$ and $t\geq 0$, 
the section centered at $x_0$ with height $t$ is the convex set
$$ S(x_0,p_0,t) := \{ x\in\R^n\, ;\, \psi(x) \leq \psi(x_0)+ p_0\cdot(x-x_0)+t\}. $$
\begin{proposition}\label{prop:secpolar}
Let $\psi:\R^n\to \R$ be a convex function.
For all $x_0\in\R^n$, $p_0\in \pa\psi(x_0)$ and $t\geq 0$ such that $S(x_0,p_0,t)$ is bounded, we have
\begin{equation} \label{eq:SoS1}
 p_0+t (S-x_0)^\circ \subset \pa\psi(S)
\quad \mbox{ and } \quad \pa\psi^* (p_0+\frac t 2 (S-x_0)^\circ) \subset S \end{equation}
and
\begin{equation} \label{eq:SoS2}
\pa\psi\left(\frac 1 2 (S+y_0) \right) \subset p_0+ 2t (S-y_0)^\circ \qquad\forall y_0\in S.
\end{equation}
\end{proposition}

When $x_0=0$ and $p_0=0$, \eqref{eq:SoS1} can be written simply as $\pa\psi^* (\frac t 2 S ^\circ) \subset S$ and $t S ^\circ \subset \pa\psi( S)$. The proof will make it clear that we also have $\pa\psi^* ((1-\eta) t S ^\circ) \subset S$ for all $\eta>0$
and even $\pa\psi^* (t S ^\circ) \subset S$ if $\psi$ is strictly convex. 
This simple relation between the sections of $\psi$ and their dual body, combined with the volume product inequality \eqref{eq:polarsup1} readily  imply in particular the estimate 
\begin{equation}\label{eq:volumesectionbd}
|\pa\psi(S)|  \, |S| \geq \frac{\omega_n^2}{(2n)^n} t^n.
\end{equation}

\begin{proof}[Proof of Proposition \ref{prop:secpolar}]
Up to replacing $\psi$ with the function $\overline \psi (x) = \psi(x_0+x) - p_0 \cdot x -\psi(x_0) $,   it is enough to prove the result when $x_0=0$, $p_0=0$ and $\psi(x_0)=0$.
We then have $\bar S= S-x_0$, $\pa\overline \psi(x) = \pa \psi (x_0+x)-p_0$ and $\pa \overline \psi^* (z) = \pa\psi^*(p_0+z) -x_0$.

Let $z \in \R^n$ be such that $z\cdot x \leq t$ for all $x\in \overline S$ (that is $z\in t\overline S^\circ$). The function
$u(x)=\overline\psi(x)-z\cdot x$  is convex and satisfies
$$
u(0)=0, \qquad 
u(x) \geq 0 \quad  \mbox{ for } x\in \pa \overline S,
$$
using the fact that $\overline\psi (x) =  t $ and $z\cdot x \leq t$ for all $x\in \pa \overline S$.

In particular, there  is at least one $y\in \overline S$ such that $u$ has a minimum at $y$, and for such a point we have $0\in \pa u(y)$ and so $z\in \pa \overline\psi(y)$. 
This implies $t \overline S^\circ \subset \pa \overline \psi(\overline S)$ which is the first inclusion in \eqref{eq:SoS1}.

Next, we note that if $z \in \R^n$ is such that $z\cdot x \leq \frac t 2$ for all $x\in \overline S$ (that is $z\in \frac t 2 \overline S^\circ$), then the same function $u$ as above satisfies $u(x)>0$ for $x\in \pa \overline S$.
In particular, any  minimum $y\in \R^n$ of $u$ must lie $S$.
Since $u$ is convex, the set of points such that $0\in \pa u$ coincides with the set of points where $u$ is minimum.
So we can claim that for all $z \in \frac t 2 S^\circ$, we have $z\in \pa\psi(x) \Rightarrow x\in S$.
This implies $\pa\psi^* (\frac t 2 S^\circ) \subset S$.

\medskip

In order to prove \eqref{eq:SoS2}, we recall that for any $y_0\in \overline S$, we have
$$ (\overline S-y_0)^\circ = \{ z\in\R^n\, ;\, z\cdot (y-y_0) \leq 1\, \mbox{ for all } y\in \overline S\}.$$
We now fix $\bar x \in \frac 1 2 (\overline S+y_0)$. Then $z\in \pa \overline \psi(\bar x)$ implies $\overline \psi(\bar x) + z\cdot (x-\bar x) \leq \overline \psi(x)$ for all $x \in \R^n$ and so
\begin{equation}\label{eq:xt} 
z\cdot (x-\bar x) \leq t \quad \mbox{ for all } x\in \overline S.
\end{equation}
In particular, if we take $x=2\bar x-y_0$ in \eqref{eq:xt} (which we can do since $2\bar x-y_0\in \overline S$ with our choice of $\bar x$), we get that 
$ z \cdot (\bar x-y_0)\leq t$. Using  \eqref{eq:xt} again, we deduce:
$$ z \cdot (x-y_0) =  z \cdot (x-\bar x)+ z \cdot (\bar x-y_0) \leq 2t \quad \mbox{ for all } x\in \overline S.
$$
We proved that 
$\pa\overline \psi(\bar x) \subset 2t (\overline S-y_0)^\circ$ for all $\bar x \in \frac 1 2 (\overline S+y_0)$, which implies \eqref{eq:SoS2}.
\end{proof}
\subsection{Volume of sections: Proof of Proposition \ref{prop:sec}-(i)}
We now prove the following proposition, which implies Proposition \ref{prop:sec}-(i):
\begin{proposition}\label{prop:St}
Let $\psi:\R^n\to \R$ be a convex function satisfying \eqref{eq:O}, \eqref{eq:ineq1}-\eqref{eq:ineq2} for measures $\mu$ and $\nu$ satisfying Assumption \ref{ass:1}.
Then there exist two constants $c$ and $C$ depending only on $\lambda$, $\Lambda$ and $n$ such that
if 
\begin{equation} \label{eq:t}
 t\geq  \max\{2L(\pa\psi(\Omega)) ,2nL(\Omega)\} \delta
 \end{equation}
then
\begin{equation}\label{eq:Stnf} 
c |S(x_0,p_0,t)|^2 \leq t^n \leq C|S(x_0,p_0,t)|^2
\end{equation}
for all sections with $x_0\in\Omega$, $p_0\in \pa\psi(x_0)$ and $t\geq 0$ such that $S(x_0,p_0,t)\subset \Omega$.
\end{proposition}

The proof of Proposition \ref{prop:St} relies on Proposition \ref{prop:secpolar} and on the lower and upper bound on the volume product proved in Proposition \ref{prop:polarinf}.

\begin{proof}
We use \eqref{eq:SoS1}  with  \eqref{eq:ineq2} to get:
\begin{equation}\label{eq:jkfg} 
\nu \left(p_0+\frac t 2 (S-x_0)^\circ\right) \leq \mu \left(\pa\psi^* \left(p_0+\frac t 2 (S-x_0)^\circ\right)\right) \leq \mu (S).
\end{equation}
Proposition \ref{prop:lL} and condition \eqref{eq:t} imply
$ \ell(S) \geq \delta $.
Since $S$ is convex and  $S\subset \Omega$,   \eqref{eq:mn3} gives
$$\mu(S) \leq \Lambda |S|.$$
Similarly, Corollary \ref{cor:IL} and  condition \eqref{eq:t} imply $ \ell(p_0+t (S-x_0)^\circ)= t \ell((S-x_0)^\circ) \geq \delta$, 
and \eqref{eq:SoS1} implies  $ p_0+t (S-x_0)^\circ \subset \pa\psi(S)\subset \O $. 
So  \eqref{eq:mn4} gives
$$\nu \left(p_0+\frac t 2 (S-x_0)^\circ\right) \geq \lambda \left|\frac t 2 (S-x_0)^\circ\right| =  \lambda\frac {t^n}{2^n} \left| (S-x_0)^\circ\right|.$$
Going back to \eqref{eq:jkfg}, we proved that
$$t^n |(S-x_0)^\circ| \leq \frac{2^n\Lambda}{\lambda} |S|.$$
Using the lower bound on the volume product \eqref{eq:polarsub1}, we deduce
$$ |S| ^2 \geq c t^n |(S-x_0)^\circ|  |S-x_0|\geq ct^n, $$
for a constant $c$ that depends only on $n$, $\lambda$ and $\Lambda$.
This is the second inequality in \eqref{eq:Stnf}.

To prove the first inequality, we use \eqref{eq:SoS2} with \eqref{eq:ineq1}
to write (for any $y\in S$):
$$\mu\left(\frac 1 2 (S+y)\right) \leq \nu\left(\pa\psi\left(\frac 1 2 (S+y) \right)\right) \leq \nu \left((p_0+ 2t (S-y) ^\circ\right)\cap\O).$$
We note that $ \frac 1 2 (S+y)\subset S \subset \Omega$  and Proposition \ref{prop:lL} together with condition \eqref{eq:t} imply
$\ell(\frac 1 2 (S+y)) = \frac 1 2 \ell(S) \geq \delta$.
Also, \eqref{eq:SoS1} implies that $p+ t (S-y) ^\circ \subset (p_0+ 2t (S-y) ^\circ )\cap \O$, and so Corollary \ref{cor:IL}  together with  condition \eqref{eq:t}  implies
$\ell ((p_0+ 2t (S-y) ^\circ)\cap \O)\geq \delta$.

We can thus use Assumption \ref{ass:2}  and get
$$
\frac{\lambda }{2^n} |S| = \lambda  \left|\frac 1 2 (S+y)\right| \leq \Lambda    \left|p+ 2t (S-y) ^\circ\right| = \Lambda 2^n t^n   \left| (S-y) ^\circ\right|
$$
and so
$$ |S| ^2 \leq C t^n |(S-y) ^\circ||S|$$
for a constant $C$ depending only on $n$, $\lambda$ and $\Lambda$.
The result follows by taking the infimum over $y$ and using the upper bound on the volume product \eqref{eq:polarsup1}.
\end{proof}
\subsection{Scaling properties of sections: Proof of Proposition \ref{prop:sec}-(ii)}
For any convex function $\psi$ and any $\tau \in(0,1)$ we have 
\begin{equation} \label{eq:Stau}
\tau S(x_0,p_0,t) \subset S(x_0,p_0,\tau t).
\end{equation}
The main result of this section is the following proposition which provides control in the other direction when $\psi$ is associated to our optimal transportation problem:
\begin{proposition}\label{prop:sec1}
Let $\psi:\R^n\to \R$ be a convex function satisfying \eqref{eq:O}, \eqref{eq:ineq1}-\eqref{eq:ineq2} for measures $\mu$ and $\nu$ such that \eqref{eq:munu1}-\eqref{eq:munu2} holds.
For any fixed $\tau\in (0,\ 1)$, there exists $\gamma\in(0,1)$, depending on  $\lambda$, $\Lambda$ and $n$ such that
if  $t$ satisfies \eqref{eq:t} and 
\begin{equation}\label{eq:t2} 
t\geq  8nL(\Omega) \delta
\end{equation}
then
\begin{equation}\label{eq:Slambda}
 S(x_0,p_0,\tau t) \subset \gamma S(x_0,p_0, t)
  \end{equation}
 for all $x_0\in\Omega$, $p_0\in \pa\psi(x_0)$ and $t\geq 0$ such that $S(x_0,p_0,t)\subset \Omega$.
\end{proposition}

First, we prove the following technical lemma for convex sets:
\begin{lemma}\label{lem:B}
Let $K$ be a compact convex set with nonempty interior such that $0\in K$.
For all $\gamma\in(0,1)$ and  $y \in K\setminus \gamma K$ we have:
$$ |K||(K-y)^\circ| \geq  \frac{c}{1-\gamma} $$
for some constant $c$ depending only on the dimension. 
\end{lemma}
Note that $|(K-y)^\circ|\to\infty$ when $y\to \pa K$  since  $(K-y)^\circ$ is unbounded when  $y\in \pa K$. So this lemma simply quantifies how large the volume product will be when $y$ is close to $\pa K$ (which corresponds to $\gamma\to1$). 

\begin{proof}
John's lemma implies that there exists a linear transform $L:\R^n\to \R^n$ such that
$ B_{1}(b) \subset L(K) \subset B_n(b)$
for some $b\in \R^n$ (we do not take $b=0$ as translations have a strong effect on the polar body). We note that $L(K-y)  = L(K)-L(y)$ and $L(y) \in L( K) \setminus\gamma L( K)$.
Since the product volume $|K^\circ |\, |K|$ is invariant under linear transformation (although not under affine transformation), it is  enough to prove the result 
under  the condition 
 \begin{equation}\label{eq:norm}
 B_{1}(b) \subset K \subset B_n(b).
\end{equation}
We now claim that 
\begin{equation}\label{eq:Sybd}
|(K-y)^\circ | \geq \frac{c}{\mathrm{dist}(y,\pa K) L(K)^{n-1}}\geq \frac{1}{1-\gamma} \frac{c}{L(K)^{n}}.
\end{equation}
To prove \eqref{eq:Sybd}, we take $y_0\in\pa K$ such that $ |y_0-y|=\mathrm{dist}(y,\pa K) =:\eta$.
Up to a rotation, we can always assume that the vector $y_0-y$ is in the direction $e_n$ (that is $y_0-y = \eta\, e_n$) and we denote $x=(x',x_n)$ a vector in $\R^n$.
By convexity of $K$ we have $ (x-y_0)\cdot (y-y_0)\geq 0$ for all $x\in K$,
and so
$$ x_n-y_n \leq\eta\qquad \forall x\in K.$$
Given $z$ such that
$ 0< z_n < 1/\eta$ and $|z'|\leq \frac{1-z_n\,\eta}{L(K)}$ we then have, for all $x\in K$:
$$ z\cdot (x-y) = z'\cdot (x'-y') + z_n (x_n-y_n) \leq | z'| L(K) + z_n \eta \leq 1.$$
We deduce that
$$ \left\{ z\in \R^n\, ;\, 0< z_n < 1/\eta, \;  |z'|\leq \frac{1-z_n\eta}{L(K)}\right\} \subset (K-y)^\circ.$$
Since the cone in the left hand side has volume $\frac{C}{\eta\,(L(K))^{n-1}}$ 
we get the first inequality in \eqref{eq:Sybd}.

To obtain the second inequality, we use the fact that $\frac y\gamma\notin K$ to get
$$\mathrm{dist}(y,\pa K) \leq \min\left\{L(K), \left|\frac y \gamma-y\right|  \right\}
\leq L(K)\min\left\{ 1, \frac{1-\gamma}{\gamma}\right\}\leq 2 L(K)(1-\gamma)
$$
where we used the fact that $0\in K$ and so $|y|\leq L(K)$, which concludes the proof of \eqref{eq:Sybd}.

Finally, \eqref{eq:Sybd} implies
$$ 
|(K-y)^\circ | |K| \geq \frac{1}{(1-\gamma)}\frac{c_n |K|}{L(K)^{n}}, 
$$ 
and 
under condition \eqref{eq:norm} we have $\frac{ |K|}{L(K)^{n}} \geq \frac{|B_1|}{2n}$ so  the result follows.
\end{proof}

Next, we prove the following refinement of \eqref{eq:SoS1}.
\begin{lemma}\label{lem:A}
Let $\psi:\R^n\to\R$ be a convex function such that $0\in \pa\psi(0)$. For $t> 0$ and $y\in S:=S(0,0,t)$, there holds:
$$ \pa\psi^* \left(\frac 1 2  \left(t-\psi(y) \right)(S-y)^\circ\right) \subset S  .$$
\end{lemma}

\begin{proof}
We can assume that $\psi(y)<t$ (otherwise the result is trivial).
Let $z\in \frac 1 2  \left(t-\psi(y) \right)(S-y)^\circ$, that is 
$$z\cdot (x-y)\leq\frac{1}{2} (t-\psi(y)) \mbox{  for all } x\in S.$$ 
The function $u(x)= \psi(x)-\psi(y) - z\cdot(x-y)$ hence satisfies that
$u(x)> 0$ for all $x\in\pa S$. 
Since $u(y)=0$ with $y\in S$ and $u$ is convex, it follows that all minima of $u$ are in the interior of $S$.
Furthermore again since $u$ is convex, the set of points $x_0$ where $0\in \pa u(x_0)$ coincides with the set of points where $u$ has a minimum.
We deduce $\{ x_0\,;\, 0\in\pa u(x_0)\} \subset S$.
Since $0\in \pa u(x_0) \Leftrightarrow z\in\pa \psi(x_0)  \Leftrightarrow  x_0 \in \pa \psi^*(z)$. This implies
$\pa\psi^*(z) \subset S$
and the result follows.
\end{proof}

\bigskip

\begin{proof}[Proof of Proposition \ref{prop:sec1}]
Up to replacing $\psi$ with the function $\overline \psi (x) = \psi(x_0+x) - p_0 \cdot x $, we see that it is enough to prove the result when $x_0=0$ and $p_0=0$.

We recall that  Lemma \ref{lem:A} implies 
$$ \pa\psi^* \left(\frac 1 2  \left(t-\psi(y) \right)(S-y)^\circ\right) \subset S, $$
and so \eqref{eq:ineq2} gives
\begin{equation}\label{eq:hjkgk} 
\nu \left(  \frac 1 2  \left(t-\psi(y) \right)(S-y)^\circ \right) \leq \mu \left( \pa\psi^* \left(\frac 1 2  \left(t-\psi(y) \right)(S-y)^\circ\right) \right) \leq \mu(S) .\end{equation}
We now want to use Assumption \ref{ass:2}
 to conclude, but this requires $\ell(S)$ and $\ell\left( \frac 1 2  \left(t-\psi(y) \right)(S-y)^\circ \right)$ to be large enough.
For the first one, we use Proposition \ref{prop:lL} with condition \eqref{eq:t} to get
$ \ell(S) \geq \delta $.
For the second one, we note that Corollary \ref{cor:IL} implies
\begin{equation}\label{eq:bhjkg}
\ell\left( \frac 1 2  \left(t-\psi(y) \right)(S-y)^\circ \right) = \frac 1 2  (t-\psi(y) )\ell( (S-y)^\circ) \geq \frac{1}{4nL(\Omega)} (t-\psi(y) ), \end{equation}
and we need to distinguish two cases.

\noindent {\bf Case 1:}  If $\frac{1}{4nL(\Omega)} (t-\psi(y) )  \leq \delta$, then \eqref{eq:t2}  implies $ t-\psi(y) \leq 4nL(\Omega)\delta \leq \frac 1 2  t $ and so
 $$ \psi(y) \geq \frac 1 2 t.$$

\noindent {\bf Case 2:} If  $\frac{1}{4nL(\Omega)} (t-\psi(y) ) \geq \delta$, then \eqref{eq:bhjkg}
gives $\ell( \frac 1 2  \left(t-\psi(y) \right)(S-y)^\circ )\geq \delta$ and so we can use
Assumption \ref{ass:2} in \eqref{eq:hjkgk} to get
\begin{align*}
\lambda \left| \frac 1 2  \left(t-\psi(y) \right)(S-y)^\circ \right| \leq  \nu \left(  \frac 1 2  \left(t-\psi(y) \right)(S-y)^\circ \right) \leq \mu(S) \leq\Lambda |S|.
\end{align*}

 Proposition \ref{prop:St} implies 
$$ \frac{\lambda }{2^n} \left(t-\psi(y) \right)^n |(S-y)^\circ | |S|  \leq  \Lambda |S|^2 \leq \Lambda C t^n$$
and so (using Lemma \ref{lem:B}) we deduce
$$
 (t-\psi(y))^n \leq   C  |S|^2(1-\gamma)\leq C  (1-\gamma ) t^n\,, \qquad \forall y\in S\setminus\gamma  S
 $$
that is
$$ 
\psi(y) \geq \left( 1- \left( C(1-\gamma)\right)^{1/n}\right) t \, ,\qquad \forall y\in S\setminus\gamma S.
$$
If  we take $\gamma = 1-\frac 1 { C 2^n}$, then
 we find $$ \psi(y) \geq \frac 1 2  t \qquad\forall y\in S\setminus \gamma S.$$  
 We can now conclude: Combining both cases, we proved that $y\notin S(\frac 1 2 t)$  for all $y\in S\setminus\gamma S$
that is $S\setminus \gamma S \subset S\setminus S(\frac 1 2 t)$ which is equivalent to
$S(\frac 1 2 t) \subset \gamma S$.
\end{proof}
\subsection{Engulfing property of sections:   Proof of Proposition \ref{prop:sec}-(iii)}
We are now able to show that the scaling property implies the engulfing property of sections through the following proposition.
\begin{proposition} \label{prop:engulf}
Let $\psi:\R^n\to\R$ be a convex function, $p_0\in\pa\psi(x_0)$ and assume that there exists $\gamma\in (0,1)$ such that 
\begin{equation}\label{eq:beta}
\gamma^{-1} S(x_0,p_0, t)\subset S(x_0,p_0, 2 t)\subset \Omega.
\end{equation}
Then there exists $\theta>1$ depending only on $n$ and $\gamma$   such that if $y \in S(x_0,p_0,t) $ then $S(x_0,p_0,t)  \subset S(y,q,\theta t)$ for all $q\in \pa\psi(y)$.
\end{proposition}

Proposition \ref{prop:sec1} implies that there exists $\gamma\in(0,1)$ such that if 
$$ t\geq \max\left \{  2n L(\pa\psi(\Omega)), 8n L(\Omega)\right\}  \delta
\qquad \mbox{ and  } \qquad
S(x_0,p_0, 2 t)\subset \Omega
$$
then
$$ \gamma^{-1} S(x_0,p_0, t)\subset S(x_0,p_0, 2 t).
$$
So the assumption of Proposition \ref{prop:engulf} are satisfied in such a case and 
Proposition \ref{prop:sec}-(iii) follows.

\begin{proof}
We fix $y\in S(x_0,p_0,t)$.
Using John's lemma, we can assume (up to a translation) that $S(x_0,p_0,t)$ is almost symmetric in the sense that
\begin{equation}\label{eq:-n}
-S \subset nS.
\end{equation}

We want to prove that $\psi(x) \leq \psi(y) + q\cdot (x-y) +\theta t$ for all $x\in S(x_0,p_0,t)$.
Since $ p_0\in \pa\psi(x_0)$, we have
\begin{align*}
 \psi(y) + q\cdot (x-y) +\theta t
 & \geq \psi(x_0) + p_0\cdot (y-x_0) + q\cdot (x-y) +\theta t\\
 & \geq \psi(x_0) + p_0\cdot (x-x_0) + (q-p_0)\cdot (x-y) +\theta t\\
 & \geq \psi(x) - t  + (q-p_0)\cdot (x-y) +\theta t
\end{align*}
for all $x\in S(x_0,p_0,t)$. So the result will follow if we can show that $(q-p_0)\cdot (x-y) \geq -(\theta-1)t $  for all $ x\in S(x_0,p_0,t)$. Equivalently, we need to show that
\begin{equation}\label{eq:pq}  
(q-p_0)\cdot (y-x) \leq Ct \quad \mbox{ for all } x\in S(x_0,p_0,t).
\end{equation}
for some constant $C$.

The fact that $q\in\pa\psi(y)$ and the definition of $ S(x_0,p_0, 2 t)$ implies
$$\psi(y) + q \cdot (x-y) \leq \psi(x) \leq \psi(x_0) + p_0\cdot (x-x_0)+ 2 t\qquad \forall x\in S(x_0,p_0,2 t)$$
while the fact that $p_0\in \pa \psi(x_0)$ gives
$$\psi(x_0) + p_0\cdot(y-x_0)\leq \psi(y).$$
Adding these two inequality yields
\begin{equation}\label{eq:qpxy}
 (q-p_0)\cdot (x-y) \leq 2 t \quad \mbox{ for all } x\in S(x_0,p_0,2 t).
\end{equation}

In order to get \eqref{eq:pq}, we take $x\in S(x_0,p_),t)$ and define $\bar x = y+ \delta (y-x) = (1+\delta)y - \delta x$.
Condition \eqref{eq:-n} implies $\bar x \in (1+\delta + \delta n) S(x_0,p_0,t)$ so, by choosing $\delta$ such that $1+\delta (1+  n) =\gamma^{-1}$, we get (using  \eqref{eq:beta}):
$$ \bar x = y+ \delta (y-x) \subset \gamma^{-1} S(x_0,p_0,t) \subset S(x_0,p_0, 2 t)$$ 
We can thus take $x=\bar x$ in \eqref{eq:qpxy} and deduce \eqref{eq:pq}  with $C=\frac{2}{\delta} =\frac{2 \gamma (1+n)}{1-\gamma} $.

\end{proof}

\appendix 

\section{Approximation of regular measures by discrete measures}\label{app:app}
In this section, we assume that $\mu=f\, dx$ and $\nu= g \, dy$ are absolutely continuous measures with density $f$ and $g$ that are bounded above and below. 
In particular, there exists an optimal transport map $T $ which is the gradient of a convex function $\psi :\Omega \to \R^n$ (see Theorem \ref{thm:brenier}).

Given $\delta>0$, we can construct approximations $\mu_\delta$ and $\nu_\delta$ which satisfies the assumptions of Theorem~\ref{thm:2}. The construction is classical and recalled here briefly:
We take a set of points $x_i\in\Omega$, $i=1,\dots ,n$ such that $\Omega \subset \cup_{i=1}^N B_\delta (x_i)$ and consider the associated Voronoi tessellations 
$$ V_i = \{ x\in \Omega\, ;\, |x-x_i| = \min_{1\leq j\leq N} |x-x_j|\}.$$
We then set 
$$ \mu_\delta := \sum_{i=1}^N f_i \delta_{x_i} , \qquad f_i := \mu(V_i).$$
We recall that $V_i \subset B_\delta(x_i)$ for all $i$ and that $ W_2(\mu,\mu_\delta) \leq \delta.$ Consequently, $\mu_\delta$ satisfies \eqref{eq:munu1}.
We proceed similarly with $\nu$, with a set of points $\{y_j\}_{j=1}^{M}$ and $ \nu_\delta := \sum_{j=1}^M g_i \delta_{y_j}$ and obtain that $\nu_\delta$ satisfies~\eqref{eq:munu2} so that Assumption~\ref{ass:1} holds.

The (discrete) optimal transport plan between $\mu_\delta$ and $\nu_\delta$ has the form
$$
\gamma_\delta = \sum_{i=1}^N\sum_{j=1}^M  \gamma_{\delta,ij} \delta_{x_i} \delta_{y_j}
$$
and if $\psi_\delta$ is an associated Kantorovich potential, we know that $\gamma_\delta$ must be supported in the subdifferential of $\psi_\delta$, that is
$$
\gamma_{\delta,ij}>0 \Rightarrow y_j \in \pa \psi_\delta (x_i).
$$
As a consequence, for any Borel set $A\subset \Omega$, we have that
\[
\mu_\delta(A)= \sum_{i,\;x_i\in A} f_i=\sum_{i,\;x_i\in A}\sum_{j,\;y_j\in \pa\psi(x_i)} \gamma_{\delta,i,j}\leq \sum_{j,\;y_j\in \pa\psi(A)} \sum_{i}\gamma_{\delta,i,j}=\nu(\pa\psi_\delta(A)),
\]
since $y_j\in \pa\psi(x_i)$ and $x_i\in A$ implies that $y_j\in \pa\psi_\delta(A)$. This proves that $\mu_\delta,\;\nu_\delta$ and $\psi_\delta$ satisfy~\eqref{eq:ineq1}. Since $\psi^*_\delta$ is a Kantorovich potential between $\nu_\delta$ and $\mu_\delta$, the same argument also yields~\eqref{eq:ineq2}. 

We have hence verified most of assumptions of Theorem~\ref{thm:2}. We only need to show that we can find $\Omega'\subset\subset\Omega$, and $\rho$ s.t. $S(x,p,\rho)\in \Omega$ for every $x\in \Omega$. This is in fact non-trivial but we can fortunately rely on some of the existing literature:

First, we recall one of the main results of \cite{li-nochetto} which can be written as follows with our notations.
\begin{theorem}\label{thm:LN}
If  $-\psi(x)$ is $\lambda$-convex (in  that case,  the triple $(\mu,\nu,\psi)$ is said to be $\lambda$-regular), then there exists a constant $C$ such that
$$
\sum_{i=1}^N f_i \mathrm{dist} \,(\na \psi(x_i) , \pa \psi_\delta(x_i)) ^2 \leq C \lambda \delta.
$$
\end{theorem}
We observe that the classical regularity results by Caffarelli, \cite{Caffarelli91,C92}, give that $\psi$ and $\psi^*$ are $C^{2,\alpha}$
when $\mu$ and $\nu$ are $C^{0,\alpha}$, strictly positive on their support and  supported on $C^2$, uniformly convex domains. In particular, $(\mu,\nu,\psi)$ is $\lambda$-regular for some $\lambda$ under such conditions and Theorem \ref{thm:LN} applies. 

We recall that $\psi$ is $C^1$, so $\pa \psi(x) = \{\na \psi(x)\} $ for all $x$ and
$\mathrm{dist} \,(\na \psi(x_i) , \pa \psi_\delta(x_i)) $ denotes the distance between the point $\na \psi(x_i)$ and the set $\pa \psi_\delta(x_i)$ (in fact the result of  \cite{li-nochetto} is slightly more precise since it controls the distance between $\pa \psi_\delta(x_i)$  and an appropriate barycenter of $\pa \psi_\delta(x_i)$.

Next, we prove the following result.
\begin{theorem}\label{thm:A2}
With the notations above, and under the assumptions of Theorem \ref{thm:LN}, assume further than $\psi^*$ is $C^{1,\alpha}$. Let $[x_1,x_2]$ be a line segment in $\Omega$ with length $L$
and denote
$$t= \max_{s\in [0,1]} (1-s)\psi_\delta (x_1) + s \psi_\delta (x_2) - \psi_\delta ((1-s)x_1 + s x_2)$$
Then there exists $C$ and $\alpha_1,\alpha_2>0$ such that 
if
$$ 
\delta \leq C \min\{L^{\alpha_1}, \frac{L^{\alpha_2}}{\lambda} \} 
$$
then 
$$ t \geq c L^{\alpha_1}.$$
\end{theorem}
Note that when $t=0$, this theorem provides a bound on the diameter of  the flat parts of $\psi_\delta$ by some power of $\delta$.
For a given $t>0$, it implies a bound on the diameter of the sections of $\psi_\delta$. Indeed if $y\in \pa S_\delta(x,p,t)$, namely
$\psi_\delta(y)= \psi_\delta(x)+p\cdot (y-x)+t$, then
\[\begin{split}
&(1-s)\psi_\delta (x) + s \psi_\delta (y) - \psi_\delta ((1-s)x + s y)\\
&\qquad=(1-s)(\psi_\delta (x)-p\cdot x) + s (\psi_\delta (y)-p\cdot y) - (\psi_\delta ((1-s)x + s y)-p\cdot((1-s)x+sy))\leq t.
\end{split}
\]
Hence, taking $\delta\lesssim t$, the previous theorem implies that $|x-y|\leq C\,t^{1/\alpha_1}$ or $|x-y|\leq C\,\delta^{1/\alpha_2}$. Consider now any $\Omega'\subset \Omega$ at a distance $d_0$ from $\pa \Omega$, with $d_0\geq C\,\delta^{1/\alpha_2}$ and  take $\rho=c\,d_0^{\alpha_1}$; we necessarily have that $S(x,p,\rho)\subset \Omega$ for any $x\in \Omega'$, which is \eqref{eq:rho1}. All the assumptions of Theorem~\ref{thm:2} are thus satisfied in this case. 

\begin{proof}[Proof of Theorem \ref{thm:A2}]
We assume (without loss of generality) that $0 =\frac{x_1+x_2}{2}$ and $\psi^\delta(x_1) =\psi^\delta (x_2)$, $\na \psi(0)=0$.
We also denote by $e$ the unit vector in the direction of $x_2$. 
Given $\theta>0$, we consider the cone of opening $\theta$ and axis $e$:
$$
C_\theta = 
\left\{ x\in\R^n\, \Big|\,  \frac {x\cdot e}{|x|} \geq 1-\theta
\right\} 
$$
and
$$
A  = \left\{ x \in C_\theta \, |\, L/8 \leq  |x|\leq L/4\right\} 
$$
We have 
$$ |A| \geq C L^n \theta ^{n-1}$$
and classical properties of convex functions  (see Lemma 2.1 in \cite{JMM} for example) imply:
$$
|z \cdot e| \leq C \frac{\|\na \psi\|_{L^\infty} }{L} (L\theta + t) \qquad \forall z\in \pa \psi^\delta(x), \quad \forall x\in A_\theta.
$$
Next, the $C^{1,\alpha}$ regularity of $\psi^*$ implies
$$ \na \psi (x) \cdot e  \geq c  L^\beta \quad \mbox{ for $x\in \Omega'\cap A $}
$$
for some $\beta>0$ if $\theta$ is small enough. 
Indeed, we have $\psi(x) \geq c |x|^{1+\beta}$ and so $\na \psi(x)\cdot \frac{x}{|x|} \geq c|x|^\beta$. We deduce
$\na \psi (x) \cdot e  \geq  c \cos \theta L^\beta - \sin \theta \|\na \psi\|_\infty $ in $A$, so we must take
$\theta = c\max\{ 1, \frac{L^\beta}{\| \na \psi\|_\infty} \} $

In particular, we have
\begin{align*}
\mathrm{dist} \,(\na \psi(x_i) , \pa \psi_\delta(x_i))  
& \geq \mathrm{dist} \,(\na \psi(x_i)\cdot e , \pa \psi_\delta(x_i)\cdot e) \\
& \geq  c(1-\theta)  L^\beta - C (\theta + \frac t L) \qquad \forall x\in A
\end{align*}
With our choice of $\theta = c\max\{ 1, \frac{L^\beta}{\| \na \psi\|_\infty} \} $, we see that
if  $t\leq c L^{1+\beta}$, then
\begin{align*}
\mathrm{dist} \,(\na \psi(x_i) , \pa \psi_\delta(x_i))  
& \geq  c  L^\beta \qquad \forall x\in A .
\end{align*}
Theorem \ref{thm:LN} implies in particular 
\begin{equation}\label{eq:kjhg}
c \left( \sum_{x_i\in A_\theta}  f_i \right) L^{2\beta}\leq 
\sum_{x_i\in A_\theta}  f_i \mathrm{dist} \,(\na \psi(x_i) , \pa \psi_\delta(x_i)) ^2 \leq C \lambda \delta.
\end{equation}
It remains to see that
$$\sum_{x_i\in A_\theta}  f_i =\mu \left( \cup_{x_i\in A} V_i \right)  \geq \mu( \frac 1 2 A) \geq c L^n \theta ^{n-1} $$
if $\delta$ is small compared to the lengths of $A$, that is  $\delta \leq \theta L \sim L^{1+\beta}$.
Inequality \eqref{eq:kjhg} then implies
$$
L^\alpha \leq  C \lambda \delta
$$
with  $\alpha = n(1+\beta) +\beta$.
If $  \delta \leq c\frac{L^\alpha }{\lambda}$, that's a contradiction.
We prove that if $\delta \leq c\min \{ \frac{L^\alpha }{\lambda},L^{1+\beta}\}$, then we must have $t\geq c L^{1+\beta}$, which implies the result.

\end{proof}

\section{Proof of Proposition \ref{prop:delta}}\label{app:A}
We will prove the following proposition, which implies in particular Proposition \ref{prop:delta}.
\begin{proposition}\label{prop:bhkh}
Given a non-negative Radon measure $\mu$ on $\R^n$, the following statements are equivalent
\item[(i)] There exists $\lambda, \Lambda>0$  and $\beta$ such that
$$ 
\lambda |B_r | \leq \mu(B_r) \leq\Lambda |B_r|, \quad \mbox{ for all ball $B_r\subset \Omega$ with $r \geq \beta \delta$} 
$$
\item[(ii)] There exists $\lambda, \Lambda>0$  and $\beta$  such that
$$ 
\lambda |K| \leq \mu(K) \leq\Lambda |K|, \quad \mbox{ for all convex  Borel $K\subset \Omega$ with $\ell (K)\geq \beta \delta$} 
$$
\item[(iii)] There exists $\lambda, \Lambda>0$  and $\beta$  such that
$$ 
\lambda |A_\delta| \leq \mu(A_\delta) \leq\Lambda |A_\delta|, 
$$
for any Borel set $A$ and with $A_\delta = \cup_{x \in A}(S_\delta(x))\subset \Omega $ where $\{S_\delta(x)\}_{x\in A}$ is a family of convex sets satisfying $x\in S_\delta$, $\ell (S_\delta)\geq \beta \delta$, and for which Vitali's covering lemma (with constant $C_*$)  can be applied.
\medskip

Furthermore, the constant $\lambda$, $\Lambda$ and $\beta$ appearing in these statements differ by a factor depending only on the dimension (and $C_*$ for the last one).

\end{proposition}

\begin{proof}
 It is clear that 
$$(iii)\Rightarrow (ii)\Rightarrow (i).$$

\medskip

We will assume that $(i)$ holds with $\beta=1$ and  prove that it implies $(ii)$ (for some $\beta>1$).
In view of John's Lemma \ref{lem:J}, it is enough to prove the result when $K=E$ is an ellipsoid with smallest radius larger than $\beta \delta$.
For such an ellipsoid, we have 
\begin{equation}\label{eq:Edelta}
 E+B_\delta  := \cup_{x\in E} B_\delta(x) \subset (1+\frac 1 \beta) E.
 \end{equation}
Indeed, we have $E= \{ x\in \R^n\, ;\, \left(\sum_{i=1}^{n}\left(\frac {x_i}{a_i} \right)^2 \right)^{1/2} \leq 1\}$ with $
a_i\geq \beta\delta$. If $y\in E+B_\delta$, there exists $x\in E$ such that $|x-y|\leq \delta$ and so Minkowski inequality yields
\[
\begin{split}
\left(\sum_{i=1}^{n} \left(\frac {y_i}{a_i} \right)^2 \right)^{1/2} 
&\leq \left(\sum_{i=1}^{n} \left(\frac {x_i}{a_i} \right)^2 \right)^{1/2} +\left(\sum_{i=1}^{n} \left(\frac {y_i-x_i}{a_i} \right)^2 \right)^{1/2} 
\leq 1 + \frac{1}{\beta \delta }\left(\sum_{i=1}^{n} \left(y_i-x_i \right)^2 \right)^{1/2} \\
&\leq 1 + \frac 1 \beta.
\end{split}
\]
Since $E$ is compact, Vitali's covering Lemma implies the existence of a covering 
 $E\subset \cup_{i=1}^k B_{3\delta}(x_i)$ with $x_i\in E$ and such that the balls $B_\delta (x_i)$ are disjoint.
 We then have
 $$
 \mu(E) \leq \sum_{i=1}^k \mu(B_{3\delta}(x_i)) \leq \Lambda  \sum_{i=1}^k |B_{3\delta}(x_i)| = \Lambda 3^n \sum_{i=1}^k |B_{\delta}(x_i)| =  \Lambda 3^n | \cup_{i=1}^k B_{\delta}(x_i)| \leq  \Lambda 3^n |E+B_\delta|. 
 $$
We emphasize that we may not have $B_{3\delta}(x_i) \subset \Omega$, but the upper bound $\mu(B_{3\delta}(x_i)) \leq \Lambda C | B_{\delta}(x_i)|$ still holds).
Finally \eqref{eq:Edelta} implies that
$$
 \mu(E) \leq  
 \Lambda 3^n (1+\frac 1 \beta)^n |E|.
 $$

To prove the other inequality, we consider the ellipsoid $\tilde E= (1-\frac 1 \beta )E = \frac{\beta-1}{\beta}E$ (here, we assume that $\beta>1$).
This ellipsoid satisfies $\ell(\tilde E) \geq (\beta-1)\delta$ and so \eqref{eq:Edelta} implies
\begin{equation}\label{eq:Edeltat}
\tilde E +B_\delta \subset (1+\frac{1}{\beta-1}) \tilde E \subset \frac{\beta}{\beta-1} \tilde E = E
\end{equation}
We now consider a covering  of $\tilde E\subset \cup_{i=1}^k B_{3\delta}(x_i)$ with $x_i\in \tilde E$ and such that the balls $B_\delta (x_i)$ are disjoint.
We then have (using \eqref{eq:Edeltat}, which guarantees in particular that $B_\delta(x_i)\subset E\subset\Omega$):
$$\mu(E) \geq \mu(\tilde E+B_\delta) \geq \mu(\cup_{i=1}^k B_\delta(x_i))= \sum_{i=1}^k \mu(B_\delta(x_i)) \geq \lambda  
\sum_{i=1}^k |B_\delta(x_i)| \geq \frac{\lambda}{3^n} \sum_{i=1}^k |B_{3\delta}(x_i)| \geq  \frac{\lambda}{3^n}|\tilde E|.$$
The definition of $\tilde E$ implies
$$ \mu(E) \geq\frac{\lambda}{3^n}|(1-\frac 1 \beta )E|  = \frac{\lambda}{3^n}(1-\frac 1 \beta )^n|E| .$$
We conclude that $(ii)$ holds, with (for example) $\beta=2$.

\medskip 

We now prove that $(ii)$ implies $(iii)$ by using Vitali's covering lemma:
There exists a universal constant $C_*$ and a set of points $x_i\in A$ such that  the $S_\delta(x_i)$ are disjoint and 
$ A_\delta \subset \cup_i C_* S_\delta(x_i)$.
We then have (using $(ii)$ since the $C_* S_\delta(x_i)$ are convex):
$$ \mu(A_\delta) \leq \sum _i \mu (C_* S_\delta(x_i)) \leq \Lambda \sum_i |C_* S_\delta(x_i)| \leq \Lambda C_* ^n \sum_i |S_\delta(x_i)| = \Lambda C_*^n    |\cup_i S_\delta(x_i)| \leq  \Lambda C_*^n |A_\delta|.$$
and 
$$\mu(A_\delta) \geq \mu(\cup_i S_\delta(x_i)) = \sum_i \mu(S_\delta(x_i)) \geq \frac{\lambda}{C_*^n} \sum_i |C_* S_\delta(x_i)| \geq \frac{\lambda}{C_*^n} |\cup_i C_* S_\delta(x_i)| \geq \frac{\lambda}{C_*^n} |A_\delta|.$$
 
\end{proof}

\section{ Restriction of the Kantorovich potential}\label{app:rest}
Let $\psi:\R^n\to(-\infty,+\infty]$ be a Kantorovich potential given by Theorem \ref{thm:OT1}.
That is $\psi$ is a convex lower-semi-continuous function and there exists an optimal transference plan $\pi$ concentrated in the graph of $\pa\psi$.
We denote $\phi = \psi^*$ its Legendre transform.
The construction below is classical (see for instance \cite{Berman}) and recalled here for the reader's convenience.

Given two sets $X$ and $Y$ such that
$ \supp \mu \subset X \subset \R^n$ and $\supp \nu \subset Y\subset \R^n$, we define
$$ 
\tilde\phi(y) = \begin{cases}
\phi (y) & \mbox{ if there exists $x\in X$ such that $y\in \pa\psi(x)$ } \\
+\infty & \mbox{ otherwise}
\end{cases}
$$
and 
$$ \tilde \psi (x)= \sup_{y \in Y } (x\cdot y -\tilde \phi(y))$$
(the function $\tilde \phi$ and $\tilde \psi$ are defined in $\R^n$).

We then have the following:
\begin{proposition}
The function $\tilde \psi$ is a Kantorovich potential for the optimal transference plan $\pi$ and satisfies
$$\tilde \psi(x) = \psi(x) \mbox{  for all $x\in \supp \mu$}$$
and $\pa\tilde\psi(\R^n) \subset  {\Gamma(Y)}$, the closed convex envelop of $Y$.
\end{proposition}

If we take $Y=\supp\nu$, we get
$ \tilde \psi(x) = \sup_{y\in \supp \nu} \left(\tilde \phi(y) + x\cdot y\right)$
for some function $\phi: \supp \nu \to \overline \R$ and  $\pa \tilde \psi(\R^n) \subset \Gamma(\supp \nu) \subset \overline \O $, where $\Gamma(\supp \nu) $ denotes the closed convex envelop of $\supp \nu$.

\begin{proof}
The fact that  $\tilde \psi$ is convex and lower-semicontinuous follows from its definition and since $\tilde \phi(y) \geq \phi(y)$, we have $\tilde \psi(x)\leq \psi(x)$.

Furthermore, if $x\in X$ is such that there exists $y\in Y$ such that $y\in \pa\psi(x)$, then
$\tilde \phi(y) = \phi(y)$ and so 
$$\tilde \psi(x) \geq x\cdot y -\tilde \phi(y) = x\cdot y -  \phi(y) =  x\cdot y -  \psi^*(y)=\psi(x) $$
(the last equality follows from \eqref{eq:subeq}).

Finally, given $(x,y) \in \Graph (\pa \psi ) \cap (X\times Y)$, we now have both $\tilde \phi(y) = \phi(y)$ and $ \tilde \psi(x) = \psi(x) $
and so
$$
\tilde \psi(x) - x\cdot y = \psi(x) - x\cdot y = -\psi^*(y) = - \phi(y) = -\tilde \phi(y) \leq \tilde \psi(z) - z\cdot y$$
for all $z\in \R^n$. This inequality can also be written as 
$ \tilde \psi(z) \geq \tilde \psi(x) + y\cdot (z-y)$ which means $y\in \pa \tilde \psi (x)$. We deduce
$$ \Graph(\pa \psi ) \cap (X\times Y) \subset \Graph(\pa \tilde \psi) .$$
Since  $\pi \in \Pi(\mu,\nu)$ is concentrated in $\supp \mu\times\supp \nu$, 
we have $\supp (\pi) \subset  \Graph(\pa \psi ) \cap (X\times Y) \subset \Graph(\pa \tilde \psi) $ and so 
$\tilde \psi$ is  a Kantorovich potential (associated to the same transference plan $\pi$).

Finally, the fact that  $\tilde \psi(x)= \psi(x)$
on the $x$-projection of $\Graph(\pa \psi ) \cap (X\times Y)$ implies that
$$ \tilde \psi(x) = \psi(x) \mbox{ in } \supp \mu.$$

\end{proof}

\bibliography{mybib}{}
\bibliographystyle{plain}

\end{document}